\documentclass[12pt,reqno]{amsart}
\usepackage{amsmath}
\usepackage{geometry}
\usepackage{amssymb, latexsym}
\usepackage{verbatim}
\usepackage{enumerate}

\usepackage{mathtools}
\usepackage[tableposition=top]{caption}
\usepackage{booktabs,dcolumn}

\newtheorem{theorem}{Theorem}[section]
\newtheorem{proposition}[theorem]{Proposition}

\newtheorem{lemma}[theorem]{Lemma}
\newtheorem{corollary}[theorem]{Corollary}
\newtheorem{conjecture}[theorem]{Conjecture}

\newcommand\R{\mathbb{R}}
\newcommand\Z{\mathbb{Z}}

\newcommand\C{\mathbb{C}}

\newcommand\eps{\varepsilon}

\newcommand\scaling{.75} 

\renewcommand{\sim}{\asymp} 

\renewcommand{\phi}{\varphi}

\begin{document}


\title[Fourier uniformity]{Fourier uniformity of bounded multiplicative functions in short intervals on average}

\author{Kaisa Matom\"aki}
\address{Department of Mathematics and Statistics \\
University of Turku, 20014 Turku\\
Finland}
\email{ksmato@utu.fi}

\author{Maksym Radziwi{\l}{\l}}
\address{ Department of Mathematics,
  Caltech, 
  1200 E California Blvd,
  Pasadena, CA, 91125}
  
\email{maksym.radziwill@gmail.com}

\author{Terence Tao}
\address{Department of Mathematics, UCLA\\
405 Hilgard Ave\\
Los Angeles CA 90095\\
USA}
\email{tao@math.ucla.edu}
\begin{abstract} Let $\lambda$ denote the Liouville function. We show that as $X \rightarrow \infty$, 
  $$
  \int_{X}^{2X} \sup_{\alpha} \left | \sum_{x < n \leq x + H} \lambda(n) e(-\alpha n) \right | dx = o ( X H) 
  $$
  for all $H \geq X^{\theta}$ with $\theta > 0$ fixed but arbitrarily small.  Previously, this was only known for $\theta > 5/8$.  For smaller values of $\theta$ this is the first ``non-trivial'' case of local Fourier uniformity on average at this scale. We also obtain the analogous statement for (non-pretentious) $1$-bounded multiplicative functions.

  We illustrate the strength of the result by obtaining cancellations in the sum of $\lambda(n) \Lambda(n + h) \Lambda(n + 2h)$ over the ranges $h < X^{\theta}$ and $n < X$, and where $\Lambda$ is the von Mangoldt function.

\end{abstract}

\maketitle


\section{Introduction}
\label{se:intro}

Let $\lambda$ denote\footnote{All the results for $\lambda$ discussed here are also applicable to the M\"obius function $\mu$ with only minor changes to the arguments; we leave the details to the interested reader.} the Liouville function, that is, a completely multiplicative function with $\lambda(p)  = -1$ at all primes $p$. Among bounded multiplicative functions, $\lambda$ plays a distinguished role since the prime number theorem is equivalent to\footnote{Our conventions for asymptotic notation are given at the end of this introduction.}
\begin{equation} \label{pnt}
\sum_{n \leq x} \lambda(n) = o(x)
\end{equation}
as $x \rightarrow \infty$,
and the Riemann Hypothesis is equivalent to 
\[
\sum_{n \leq x} \lambda(n) = O_\varepsilon(x^{1/2+\varepsilon}) \quad \text{for all $\varepsilon > 0$.}
\]
A far reaching generalization of \eqref{pnt} is \textit{Chowla's conjecture} \cite{Chowla-book}, according to which, for any sequence of distinct integers $h_1, h_2, \ldots, h_k$, one has
\begin{equation} \label{conj}
\sum_{n \leq x} \lambda(n + h_1) \dotsm \lambda(n + h_k) = o(x)
\end{equation}
as $x \rightarrow \infty$, 
where we adopt the convention that $\lambda(n) = 0$ for $n \leq 0$.  Because of the equivalence of \eqref{pnt} and the prime number theorem, Chowla's conjecture
is frequently viewed as a  ``higher order'' prime number theorem.

In recent years there has been a substantial amount of progress on Chowla's conjecture.  Following the work of the first two authors \cite{MR} the authors established in \cite{MRT} an averaged form\footnote{By applying H\"older's inequality to \eqref{avg}, it is also possible to obtain an averaged version of \eqref{conj} over all shifts $h_1, \ldots, h_k$; see \cite{MRT} for details.} of this conjecture in the case $k = 2$, namely,
\begin{equation} \label{avg}
\sum_{|h| \leq H} \Big | \sum_{n \leq x} \lambda(n) \lambda(n + h) \Big | = o( H x)
\end{equation}
provided that $H \rightarrow \infty$ as $x \rightarrow \infty$; see also \cite{balog, kawada, kawada2, mikawa, green-tao, fk, mrt-div} for some other averaged forms of Chowla's conjecture (as well as the closely related Elliott and Hardy-Littlewood conjectures).
An equivalent form of \eqref{avg} (for related discussion, see~\cite{TaoEq}) states that
\begin{equation} \label{maines}
\sup_{\alpha} \int_{X}^{2X} \Big | \sum_{x < n \leq x + H} \lambda(n) e(- \alpha n) \Big | d x = o( H X)
\end{equation}
provided that $H \rightarrow \infty$ as $X \rightarrow \infty$. 
The estimate \eqref{maines} along with the \textit{entropy decrement argument} was used by the third author \cite{Tao} to establish a logarithmically averaged version of Chowla's conjecture, that is,
$$
\sum_{n \leq x} \frac{\lambda(n) \lambda(n + h)}{n} = o(\log x)
$$
as $x \rightarrow \infty$, for any fixed integer $h \neq 0$. 
Subsequently for odd $k$, the third author and Ter\"av\"ainen \cite{TaoTeravainen} used the entropy decrement argument and the Gowers uniformity of the ($W$-tricked) von Mangoldt function (but avoiding the use of~\eqref{maines}), to show that
\begin{equation} \label{chowla}
\sum_{n \leq x} \frac{\lambda(n + h_1) \ldots \lambda(n + h_k)}{n} = o(\log x) 
\end{equation}
as $x \rightarrow \infty$,
for any distinct integers $h_1, \ldots, h_k$ and $k$ \textit{odd}. Their argument only partially generalizes to arbitrary multiplicative functions (see \cite{TaoTeravainenGeneral}); in the case of the Liouville function, it relies crucially on the assumption that $k$ is odd.

In order to establish \eqref{chowla} for all $k$ it is necessary to establish (the logarithmically averaged version of) what we call the \textit{local (higher order) Fourier uniformity conjecture} (see \cite{TaoEq}).

\begin{conjecture}[Local higher order Fourier Uniformity] \label{localconj}
  Let $s \geq 0$. Let $G \backslash \Gamma$ be an $s$-step nilmanifold. Let $F:G \backslash \Gamma \rightarrow \mathbb{C}$ be Lipschitz continuous and let $x_0 \in G \backslash \Gamma$. Then
  $$
  \int_{X}^{2X} \sup_{g \in G} \left | \sum_{x < n \leq x + H} \lambda(n) F(g^{n - \lfloor x \rfloor} x_0) \right | dx = o( H X) 
  $$
  as soon as $H \rightarrow \infty$ with $X \rightarrow \infty$. 
\end{conjecture}

We refer to \cite{GreenTao} for the definition of the terms above, however we will not need these notions in this paper.  Informally, the conjecture asserts that on most short intervals, $\lambda$ does not exhibit significant correlation with any $s$-step nilsequence (of bounded complexity).   
The estimate \eqref{maines} proven in \cite{MRT} essentially corresponds to the case $s = 0$ of Conjecture \ref{localconj}; this is currently the only case of the conjecture that is completely settled.

In this paper we make a first step in going beyond the case of $s = 0$ and establish the case $s = 1$ of Conjecture \ref{localconj} when $H = X^{\theta}$ with $\theta > 0$ fixed but otherwise arbitrarily small. Let us first re-state our main result for the Liouville function in a more elementary fashion.

\begin{theorem}[Local Fourier Uniformity for $s = 1$ at scale $X^{\theta}$] \label{thm:liouville}
  Let $\theta \in (0,1)$ be given and set $H = X^{\theta}$. Then
  $$
  \int_{X}^{2X} \sup_{\alpha} \Big | \sum_{x < n \leq x + H} \lambda(n) e(-\alpha n) \Big | d x = o( X H). 
  $$
  as $X \rightarrow \infty$. 
\end{theorem}

We restrict attention here to the regime $\theta \in (0,1)$, since the case $\theta \geq 1$ follows from the classical work of Davenport \cite{davenport} (and see  \cite{gt-mobius2}, \cite{gt-mobius} for the $s=2$ and $s > 2$ cases respectively of Conjecture \ref{localconj} for this range of $\theta$).
Informally, Theorem \ref{thm:liouville} asserts that on most intervals of the form $[x, x+x^\theta]$, the Liouville function $\lambda(n)$ does not exhibit singificant correlation with linear phases $e(\alpha n)$; it can easily be shown to imply the $s=1$ case of Conjecture \ref{localconj} in the range $H \geq X^\theta$ by approximating the $1$-step nilsequence $n \mapsto F(g^{n - \lfloor x \rfloor} x_0)$ by a Fourier series.

Previously, Theorem \ref{thm:liouville} was known unconditionally only for $\theta > 5/8$ from the work of Zhan~\cite{Zhan}, who showed that as $X \rightarrow \infty$ the bound $\sum_{x < n \leq x + H} \lambda(n) e(-\alpha n) = o(X H)$ holds pointwise in $x \in [X,2X]$ for $H > X^{5/8 + \varepsilon}$. It is likely that our method can be pushed to reach $H = \exp((\log x)^{1 - \delta})$ for some $\delta > 0$, and conditionally on the Riemann Hypothesis one should in principle be able to reach $H = (\log X)^{\psi(X)}$ for any function $\psi(X)$ going to infinity arbitrarily slowly with $X$, although this may require a more careful reworking of the arguments here.  It may be possible to extend the methods to this paper to also cover the $s>1$ case (again with $H = X^\theta$ for any fixed $\theta>0$); we plan to investigate this direction in future work.

Theorem \ref{thm:liouville} allows us to obtain cancellations in rather general triple correlations such as those of the form $\lambda(n) a(n + h) b(n + 2h)$, for sequences $a(\cdot)$ and $b(\cdot)$ for which sharp sieve majorants can be constructed. We illustrate the flavor of these results in the corollary below. 

\begin{corollary} \label{cor:correlations}
  Let $\theta \in (0,1)$ be given. Let $H = X^{\theta}$. Then
  $$
  \sum_{|h| \leq H} \left( 1 - \frac{|h|}{H} \right ) \sum_{n \leq X} \lambda(n) \Lambda(n + h) \Lambda(n + 2h) = o(H X) 
  $$
  as $X \rightarrow \infty$. 
\end{corollary}


Interestingly we are unable to obtain an asymptotic for
$$
\sum_{|h| \leq H} \left (1 - \frac{|h|}{H} \right ) \sum_{n \leq X}  \Lambda(n + h) \Lambda(n + 2h)
$$
for this range of $H$, since this latter problem is essentially equivalent to evaluating asymptotically $\sum_{x \leq n < x + H} \Lambda(n)$ for almost all $x \leq X$. The best result in this direction allows one to take $H > X^{1/6 - \eps(X)}$ with $\eps(X)$ tending to zero arbitrarily slowly as $X \rightarrow \infty$. This is due to Zaccagnini \cite{Zaccagnini}, building on ideas of Heath-Brown \cite{Heath-Brown} and Huxley \cite{Huxley}.  Thus, Corollary \ref{cor:correlations} gives a rare example of a sum involving the Liouville function that becomes harder to control when the Liouville function is removed!

In a subsequent paper we will obtain variants of Theorem \ref{thm:liouville} and Corollary \ref{cor:correlations} for unbounded multiplicative functions such as the divisor function or coefficients of automorphic forms. This will improve (in the $H$ aspect) earlier results of Blomer \cite{Blomer} that allowed one to take $H = X^{1/3 + \varepsilon}$ in the triple correlations of the divisor function; however, in contrast to the results of \cite{Blomer}, we will not obtain power-savings in the error terms.

Theorem \ref{thm:liouville} can in fact be generalized to almost all multiplicative functions $f: \mathbb{N} \rightarrow \mathbb{C}$ with $|f| \leq 1$ (we call such multiplicative functions \emph{$1$-bounded}). There is however one obstruction: if $f(n) = n^{it} \chi(n)$ with $|t| \leq \varepsilon X^2 / H^2$ for a small absolute constant $\eps>0$ and $\chi$ a Dirichlet character of bounded conductor $q$, then one can check (using a Taylor expansion) that
\begin{equation} \label{failure}
\int_{X}^{2X} \sup_{\alpha} \left | \sum_{x < n \leq x + H} f(n) e(-\alpha n) \right | dx \gg X H.
\end{equation}
In fact for each $x \in [X,2X]$ one can set $\alpha$ equal to $\frac{t}{x} + \frac{a}{q}$ for some integer $a$ coprime to $q$, and then $f(n) e(-\alpha n) \approx \chi(n) e(-an/q) x^{it}$ will typically have a mean of magnitude $\asymp 1/\sqrt{q}$ if $\chi$ is primitive.

Therefore the proper analogue of Theorem \ref{thm:liouville} can only hold for multiplicative functions $f$ that ``do not pretend'' to be any multiplicative function of the form $n \mapsto n^{it} \chi(n)$ with $|t| \leq X^{2} / H^2$ and $\chi$ of bounded conductor. To quantify this notion of ``pretentiousness'', we follow Granville and Soundararajan \cite{GS} and introduce the distance function
  $$
  \mathbb{D}(f; X; Q) \coloneqq \inf_{\substack{\chi \mod{q} \\ q \leq Q \\ |t| \leq X}} \Big ( \sum_{p \leq X} \frac{1 - \mathrm{Re} (f(p) p^{it} \chi(p))}{p} \Big )^{1/2}.
  $$
In particular $\mathbb{D}(f; X; Q)$ is small whenever $f$ is close to $n \mapsto n^{it} \chi(n)$ with\footnote{The role of the parameter $X$ here is mostly to control the size of $t$.  It is not important that the sum over $p$ runs up to $X$; it could run up to $X^B$ for any $B > 0$, since primes in $(X^\alpha, X^\beta]$ contribute only $O_{\alpha, \beta}(1)$ to the distance.} $|t| \leq X$ and $\chi$ of conductor $\leq Q$.

Our main theorem, stated below, confirms that $n \mapsto n^{it} \chi(n)$ with $|t| \leq X^2 / H^{2 - o(1)}$ and $\chi$ of bounded conductor are essentially the only examples of $1$-bounded multiplicative functions for which \eqref{failure} can happen. 

\begin{theorem}[Main theorem]\label{main-thm}
  Let $\theta \in (0,1)$ and $\eta > 0$.
  Let $f: \mathbb{N} \rightarrow \mathbb{C}$ be a multiplicative function with $|f| \leq 1$.
  Suppose that, for $H = X^{\theta}$, we have
  $$
  \int_{X}^{2X} \sup_{\alpha} \left | \sum_{x < n \leq x + H} f(n) e(-\alpha n) \right | d x \geq \eta H X.
  $$
  Then, for any $\rho \in (0, \frac{1}{8})$, 
  $$
\mathbb{D}(f; X^2 / H^{2 - \rho}; Q) \ll_{\eta, \theta, \rho} 1 
$$
for some $Q \ll_{\eta, \theta, \rho} 1$. 
\end{theorem}
Theorem \ref{main-thm} yields an analogous result to Corollary \ref{cor:correlations} for general multiplicative functions.  
Without going into full generality we highlight that the result holds for correlations $f(n) a(n + h) b(n + 2h)$ and sequences
$a(n)$, $b(n)$ that admit sharp sieve majorants. We illustrate this principle in the corollary below. 
\begin{corollary} \label{maincor}
  Let $\theta \in (0, 1)$. 
  Let $f : \mathbb{N} \rightarrow \mathbb{C}$ be a $1$-bounded multiplicative function.
  Suppose that $a(n), b(n)$ are sequences such that $a(n), b(n) \ll 1 + \Lambda(n)$ for all $n \geq 1$.
  
  If
  $$
  \left | \sum_{|h| \leq H} \left ( 1 - \frac{|h|}{H} \right ) \sum_{n \leq X} f(n) a(n + h) b(n + 2h) \right |  > \eta X H
  $$
  with $H = X^{\theta}$, then for any $\rho \in (0, \frac{1}{8})$, 
  $$
  \mathbb{D}(f; X^2 / H^{2 - \rho}; Q) \ll_{\eta, \theta, \rho} 1 
  $$
  for some $Q \ll_{\eta, \theta, \rho} 1$. 

The claim holds also when $f(n) a(n + h) b(n + 2h)$ is replaced by $a(n) f(n + h) b(n + 2h)$ or by $a(n) b(n + h) f(n + 2h)$.
\end{corollary}

We give the short derivation of Corollary \ref{maincor} from Theorem \ref{main-thm} in Section \ref{se:triplecorr}. It is possible to extend Corollary \ref{maincor} to sequences $b(n)$ or $a(n)$ equal to a multiplicative function $h: \mathbb{N} \rightarrow \mathbb{C}$ such that $|h(n)| \leq d_k(n)$ for all $n \geq 1$ and $k \geq 1$ a fixed integer. Since we will obtain a stronger result along these lines in a follow-up paper we do not include the details here. 


It is immediate from Corollary \ref{maincor} that given $1$-bounded multiplicative functions $f_1, f_2, f_3$, the correlations
$$
\sum_{|h| \leq H} \left ( 1 - \frac{|h|}{H} \right ) \sum_{n \leq X} f_1(n) f_2(n + h) f_3(n + 2h)
$$
vanish asymptotically whenever at least one of the $f_i$ is non-pretentious in the sense that $\mathbb{D}(f_i;X,Q) \to \infty$ as $X \to \infty$ for each $Q$. In the remaining case that all of the $f_i$ are pretentious, an asymptotic for the correlations, without an average over $h$, can be obtained using the method of \cite{Klurman} (see also the references therein).

\subsection{An overview of the proof}

We now describe in some detail the main ideas behind the proof of Theorem \ref{main-thm}.  Our presentation here is somewhat oversimplified to avoid technical issues; the actual rigorous argument will not quite follow the outline given here, but uses essentially the same ideas, despite being arranged slightly differently to resolve these technicalities.

First we notice that, by the ``analytic'' large sieve inequality (or more precisely, a maximal version of this inequality due to Montgomery \cite{Montgomery2}), given an interval $I = (x, x+H]$, there are at most $\ll \eta^{-2}$ values $\alpha_I$ (modulo $1$ and up to perturbations by $O(1/H)$) for which
\begin{equation}\label{jin}
\left | \sum_{n \in I'} f(n) e(-\alpha_I n) \right | > \eta H
\end{equation}
for some $I' \subset I$; see Lemma \ref{large}.  For sake of this informal presentation, one can pretend that in fact there is only \emph{one} such value $\alpha_I$ (modulo $1$ and perturbations by $O(1/H)$).  Thus, if there are two subintervals $I'_1,I'_2$ of $I$ (or of a slight dilate of $I$) and two frequencies $\alpha_{I,1}, \alpha_{I,2}$ obeying \eqref{jin}, one can pretend that
\begin{equation}\label{aleoh}
 \alpha_{I,1} = \alpha_{I,2} + O\left(\frac{1}{H}\right) \pmod{1}.
\end{equation}

Informally, the estimate \eqref{jin} asserts that $f$ exhibits significant oscillation at frequency $\alpha_I$ on the interval $I$ (or a large subinterval of this interval).  We depict this situation schematically in Figure \ref{fig:interval}.  In the schematic depictions we are pretending that if two such intervals $I_1, I_2$ overlap (or are very near to each other), then their associated frequencies $\alpha_{I_1}, \alpha_{I_2}$ are close modulo $1$ in the sense of \eqref{aleoh}.

\begin{figure} [t]
  \centering
    \scalebox{\scaling}{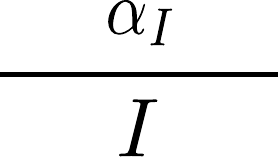}
\caption{A schematic depiction of an interval $I$ in which $f$ oscillates with frequency $\alpha_I$.}
\label{fig:interval}
\end{figure}

At this point we point out a key example: if $f(n) = n^{it}$ for some $t = o( X^2/H^2)$, some Taylor expansion of the phase $n \mapsto t \log n$ of $f$ in $I$ reveals that one has the above inequality for some $\eta \gg 1$ and $\alpha_I = \frac{t}{x_I}$, where $x_I$ denotes the starting point of $I$.  Thus, under the hypotheses of Theorem \ref{main-thm}, we expect $\alpha_I$ to vary in $I$ in a manner which is ``inversely proportional'' to the location of $I$ in some sense.  The bulk of our argument is devoted to rigorously verifying some version of this expectation; the main obstacle to overcome arises from the fact that $\alpha_I$ is only determined up modulo $1$ and up to perturbations by $O(1/H)$. 

Next, we recall an observation of Elliott \cite{ElliottBook} that by an application of the arithmetic large sieve inequality for a big set of primes $\mathcal{P} = \mathcal{P}_I \subset [2, H^{1/2}]$, we have, for all $p \in \mathcal{P}$,  
\begin{equation} \label{tk} 
\frac{1}{p} \left | \sum_{n \in I} f(n) e(-\alpha_{I} n) \right | \approx \left | \sum_{n \in I / p} f(n) e(-\alpha_{I} np) \right |;
\end{equation}
see Proposition \ref{msd}.
To make things simpler we proceed in this outline as if the approximation \eqref{tk} held for \emph{all} primes $p \asymp P$ with $P \coloneqq H^{\varepsilon}$ and some small absolute constant $\eps>0$.   Informally, \eqref{tk} asserts that if $f(n)$ behaves like a constant multiple of $e(\alpha_I n)$ for $n \in I$, then $f(m)$ behaves like a constant multiple of $e(\alpha_I m p)$ for $m \in I/p$.  Heuristically, this follows from the relationship $f(mp) = f(p) f(m)$ (at least when $m$ is coprime to $p$). We describe the estimate \eqref{tk} schematically by the diagram in Figure \ref{fig:intervalp}. Note that this is consistent with the previous heuristic that $\alpha_I$ should be inversely proportional to the location of $I$.

\begin{figure} [t]
  \centering
    \scalebox{\scaling}{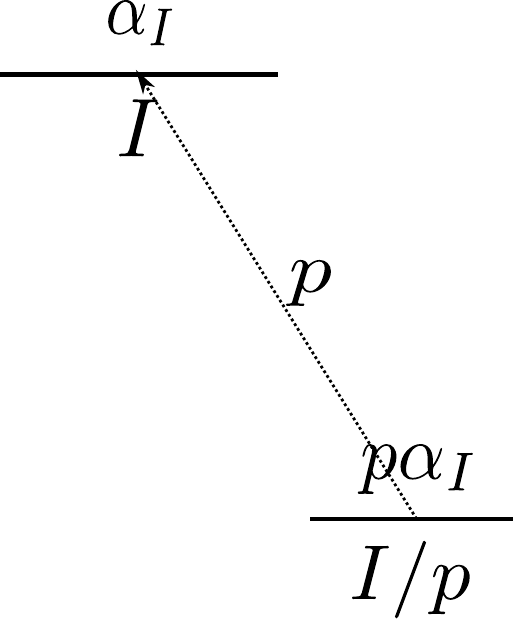}
\caption{If $f$ oscillates at frequency $\alpha_I$ on a long interval $I$, then one expects $f$ to oscillate at frequency $p \alpha_I$ on the shorter interval $I/p$. (The intervals are not drawn to scale.)  The dotted arrow indicates the fact that if one dilates $I/p$ by $p$ one returns to the interval $I$.}
\label{fig:intervalp}
\end{figure}

By the hypotheses of Theorem \ref{main-thm}, we have some frequencies $\alpha_{(x, x + H]}$ for which
$$
\int_{X}^{2X} \left | \sum_{x < n \leq x + H} f(n) e(-\alpha_{(x, x + H]} n) \right | dx \geq \eta X H,
$$
and hence by a pigeonhole principle argument, we can find a large ($\asymp X/H$) set of disjoint intervals $I$ of length $H$ in $[X,2X]$ for which \eqref{jin} holds (after modifying $\eta$ slightly).  From this, \eqref{tk}, and the Cauchy-Schwarz inequality, we will be able to locate a large set of quadruples $(I, J, p, q)$ with $I$ and $J$ disjoint intervals of length $H = X^{\varepsilon}$ for which
\begin{equation} \label{2big}
\left | \sum_{n \in I} f(n) e(-\alpha_{I} n) \right | \gg H \text{ and } \left | \sum_{n \in J} f(n) e(-\alpha_{J} n) \right | \gg H
\end{equation}
and $p,q \asymp P = H^{\varepsilon}$ are primes for which \eqref{tk} holds and such that $I / p \cap J / q \neq \emptyset$; see Figure \ref{fig:intervalpq}. 

\begin{figure} [t]
  \centering
    \scalebox{\scaling}{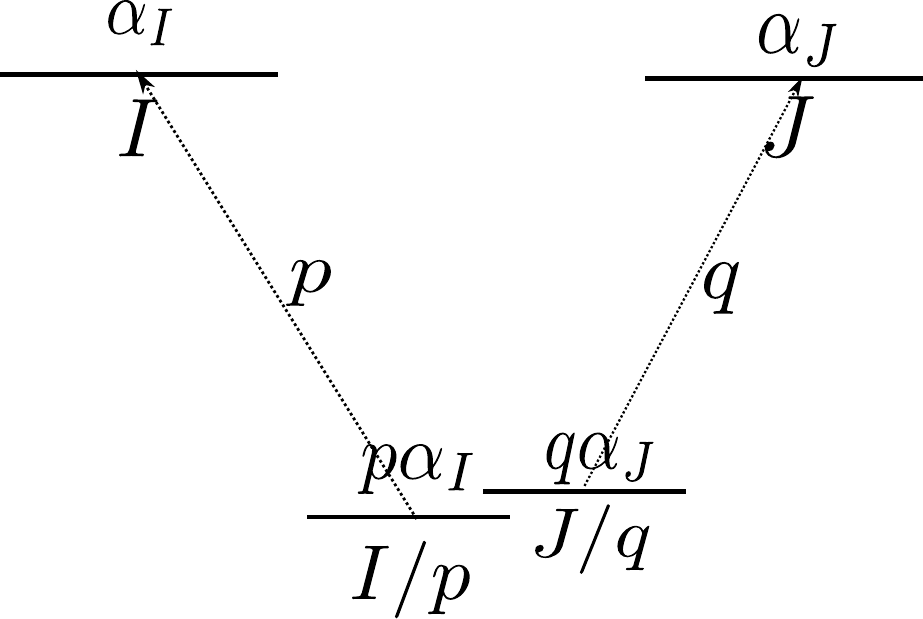}
\caption{If $f$ oscillates at frequencies $\alpha_I, \alpha_J$ on $I$, $J$ respectively, and $I/p$ overlaps $J/q$, then one expects $p \alpha_I$ and $q \alpha_J$ to often be close to each other (modulo integers).}
\label{fig:intervalpq}
\end{figure}

Since the intervals $I/p$ and $J/q$ are nearby and the frequencies $p \alpha_I$, $q \alpha_{J}$ lead to very large values of the short trigonometric polynomial supported respectively on $I/p$ and $J/q$, we conclude from \eqref{aleoh} that these frequencies lie (modulo $1$ and up to perturbations by $O(P/H)$) in a bounded set of $\ll 1$ frequencies. In particular by the pigeonhole principle it follows that, for a positive proportion of disjoint intervals $I,J$ of length $H$ and primes $p, q$ of size $P = H^{\varepsilon}$ with $I/p \cap J / q \neq \emptyset$, we have the fundamental approximate equation
\begin{equation}
\label{eq:palphaI}
p \alpha_I \equiv q \alpha_J + O(P/H) \pmod{1} 
\end{equation}
relating the frequencies $\alpha_I, \alpha_J$ associated to these intervals.	
The number of such quadruples $(I,J, p, q)$ is $\asymp (X / H) \cdot (P / \log P)^2$, since once $I, p, q$ are chosen, $J$ is essentially determined by $I / p \cap J / q \neq \emptyset$. 

It would be nice if the congruence \eqref{eq:palphaI} held $\pmod{p}$ rather than just $\pmod{1}$, as one could then profitably divide by $p$.  Fortunately, by the Chinese remainder theorem there exists a (potentially very large!) integer $k$ depending on $J$ and $q$ such that if we redefine $\alpha_J$ by shifting it by $k$, then we do indeed have 
$$
p \alpha_{I} \equiv q \alpha_J + O\left(\frac{P}{H}\right) \pmod{p}
$$
or equivalently
$$\alpha_{I} \equiv \frac{q}{p} \cdot \alpha_J + O\left(\frac{1}{H} \right) \pmod{1}$$
for all $p \asymp P$, with $p \neq q$. Importantly, shifting $\alpha_J$ by $k \in \mathbb{Z}$ maintains the property \eqref{2big}, no matter how large $k$ is. The dependence of the integer $k$ on $q$ is a bit problematic; however let us suppose for sake of discussion that $k$ is independent of $q$ (we essentially end up achieving this through a different argument that involves two consecutive applications of the arithmetic large sieve). Then applying Cauchy-Schwarz we conclude that, for a positive proportion of intervals $J_1,J_2$ and primes $q_1, q_2 \asymp P$  
with\footnote{More precisely, $\frac{J_1}{q_1}$ and $\frac{J_2}{q_2}$ will both intersect a third interval $\frac{I}{p}$, but this is almost the same as requiring that these intervals intersect each other, as they are all of comparable size; see Figure \ref{fig:intervalpqq}. For sake of this discussion, we ignore this technical distinction.} $\frac{J_1}{q_1} \cap \frac{J_2}{q_2} \neq \emptyset$, we have
\begin{equation} \label{conclusion}
\frac{q_1}{p}  \alpha_{J_1} \equiv \frac{q_2}{p}  \alpha_{J_2} + O \left ( \frac{1}{H} \right  ) \pmod{1}
\end{equation}
for many primes $p \asymp P$.  This is essentially the outcome of Section \ref{se:init}, though the argument there proceeds using a somewhat different arrangement of the above ingredients, most notably in that the prime $p$ ends up being at a different scale to the primes $q_1,q_2$, and the intervals $J_1, J_2$ have length a bit less than $H$ (and are located at spatial scales a bit less than $X$).  For sake of this discussion we assume that for the data $J_1,J_2,q_1,q_2$ as above, the relation \eqref{conclusion} holds for \emph{all} $p \asymp P$, not just for many such primes. We depict this relationship in graph theoretic language by connecting $J_1$ to $J_2$ by an edge which we label by the ratio $\frac{q_2}{q_1}$ of the primes needed to get from $J_1$ to (the vicinity of) $J_2$ by multiplication; see the dashed line in Figure \ref{fig:intervalpqq}.  The resulting graph ${\mathcal G}$ is essentially undirected (except that if one wanted to get from $J_2$ to $J_1$ one would use the label $\frac{q_1}{q_2}$ rather than $\frac{q_2}{q_1}$) and multiplicity-free (the ratios $\frac{q_2}{q_1}$ for $q_1 \neq q_2$ are all well separated from each other, so each pair $J_1,J_2$ of distinct intervals may be connected by at most one such ratio).

\begin{figure} [t]
  \centering
    \scalebox{\scaling}{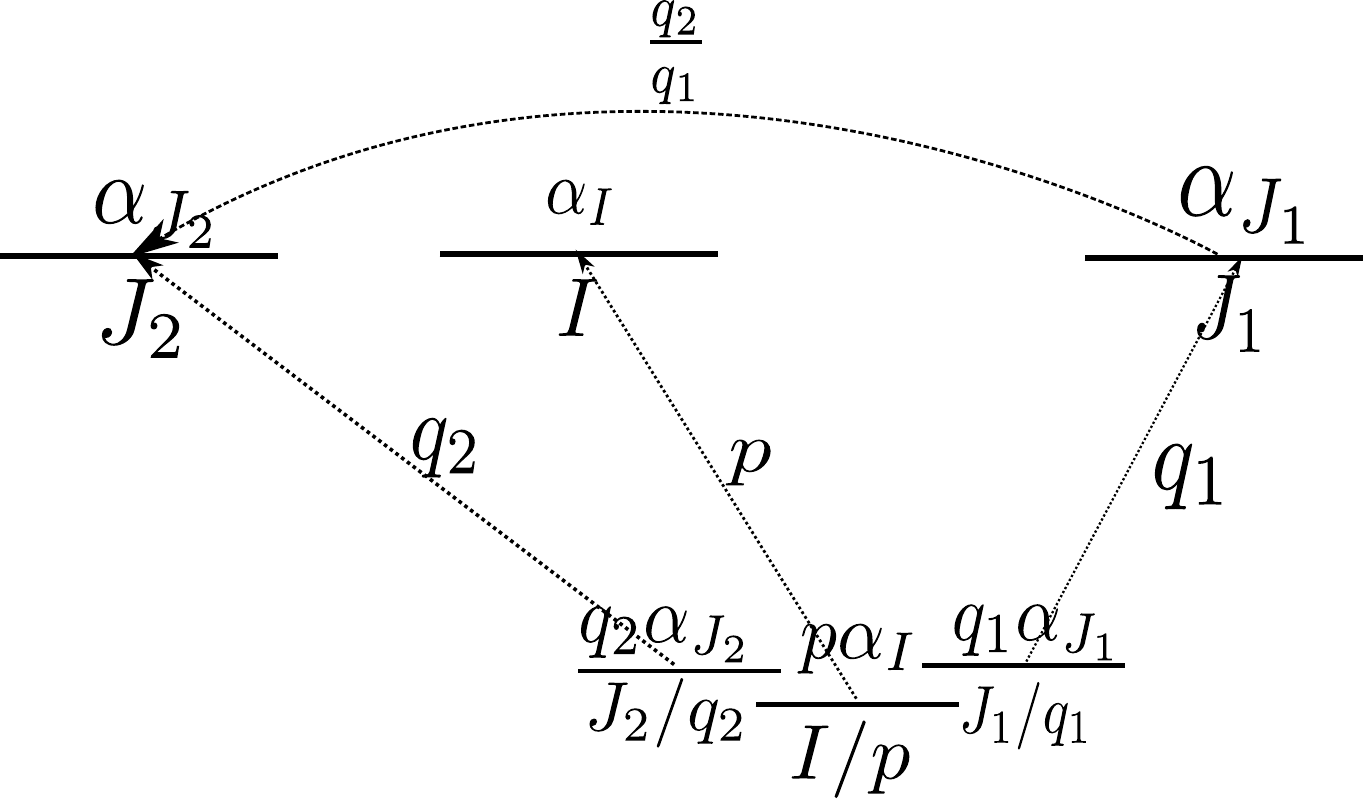}
\caption{A pair of intervals $J_1, J_2$ and primes $q_1, q_2$ such that $J_1/q_1$ and $J_2/q_2$ both meet $I/p$ for many pairs $(I,p)$, and are thus close to each other.  The frequencies $\alpha_{J_1}, \alpha_{J_2}$ have been adjusted by suitable integers so that $\frac{q_1}{p} \alpha_{J_1}, \frac{q_2}{p}  \alpha_{J_2}$ are both close (modulo integers) to $\alpha_I$, and thus also close to each other (again modulo integers).  When one has the above diagram for all (or most) $p \asymp P$, we draw a dashed line from $J_1$ to $J_2$ as indicated.  Note that if one dilates $J_1$ by $\frac{q_2}{q_1}$ then one will end up with an interval close to $J_2$.}
\label{fig:intervalpqq}
\end{figure}

Notice that the number of intervals $J_1, J_2$ and primes $q_1,q_2 \asymp P$ constructed above is $\asymp (X / H) \cdot (P / \log P)^2$; thus the graph ${\mathcal G}$ described above has $\asymp X/H$ vertices and average degree $\asymp (P/\log P)^2$. We begin Section \ref{se:local} by applying H\"older's inequality on ${\mathcal G}$ in a way that is motivated by Sidorenko's conjecture (see \cite{sidorenko}). We choose $k$ to be the first even integer for which 
$$
\left ( \frac{P}{\log P} \right )^{2k-2}  \geq \left(\frac{X}{H}\right)^2.
$$
Because of our hypotheses $H = X^\theta$ and $P = H^\eps$, we can take $k$ to be independent of $X$.  Roughly speaking, $k$ is the first integer at which we expect to see a very large number of non-trivial cycles of length $2k$ in the graph ${\mathcal G}$ .  After many applications of H\"older's ineqality, we can conclude that, for a positive proportion of disjoint intervals $I_1, J_1 \subset [X, 2X]$ of length $H$ and primes $p_1, q_1 \asymp P$ with $I_1 / p_1 \cap J_1 / q_1 \neq \emptyset$, there exist
\[
\gg \frac{H^2}{X^2} \left(\frac{P}{\log P}\right)^{4k} \gg 1
\]
``chains'' of intervals $I_2 \ldots, I_k, J_2 \ldots, J_k \subset [X, 2X]$ of length $H$ and primes 
\[
p_{1,1}, \ldots, p_{k, 1}, p_{1,2}, \ldots, p_{k, 2}, q_{1,1}, \ldots, q_{k, 1}, q_{1,2}, \ldots, q_{k, 2} \asymp P
\]
such that, for all $\ell = 1, 2, \ldots, k$,
\begin{equation} \label{eq:nearby}
\frac{I_{\ell}}{p_{\ell, 1}} \cap \frac{I_{\ell + 1}}{p_{\ell, 2}} \neq \emptyset, \quad \frac{J_{\ell}}{q_{\ell, 1}} \cap \frac{J_{\ell + 1}}{q_{\ell, 2}} \neq \emptyset
\end{equation}
and furthermore the approximate identities
\begin{equation} \label{eq:eq}
\begin{split}
\frac{p_{\ell, 1}}{p} \alpha_{I_\ell} & \equiv \frac{p_{\ell, 2}}{p} \alpha_{I_{\ell + 1}} + O \left ( \frac{1}{H} \right ) \mod{1} \\ \frac{q_{\ell, 1}}{p} \alpha_{J_{\ell}} & \equiv \frac{q_{\ell, 2}}{p} \alpha_{J_{\ell + 1}} + O \left ( \frac{1}{H} \right ) \mod{1} \\ 
\frac{p_1}{p}  \alpha_{I_1} &\equiv \frac{q_1}{p}  \alpha_{J_1} + O \left ( \frac{1}{H} \right  ) \pmod{1}
\end{split}
\end{equation}
hold for all $p \asymp P$, where we adopt the cyclic conventions $I_{k+1}=I_1, J_{k+1}=J_1$. The above set of relationships corresponds to two cycles of length $k$ in ${\mathcal G}$ connected by a further edge in ${\mathcal G}$; see Figure \ref{fig:bicycle}. The choice of $k$ is just large enough to ensure that the configuration in this figure will usually be non-degenerate in the sense that the primes $p_{1,1},\dots,q_{k,2},p_1,q_1$ that arise are all distinct for most of the configurations.  Since the primes $p$ in our case are of size $P = H^{\varepsilon} = X^{\varepsilon \theta}$, it suffices to take $k$ bounded in terms of $\varepsilon,\theta$ to guarantee the existence of a large number of such chains. 

\begin{figure}
  \centering
  \scalebox{\scaling}{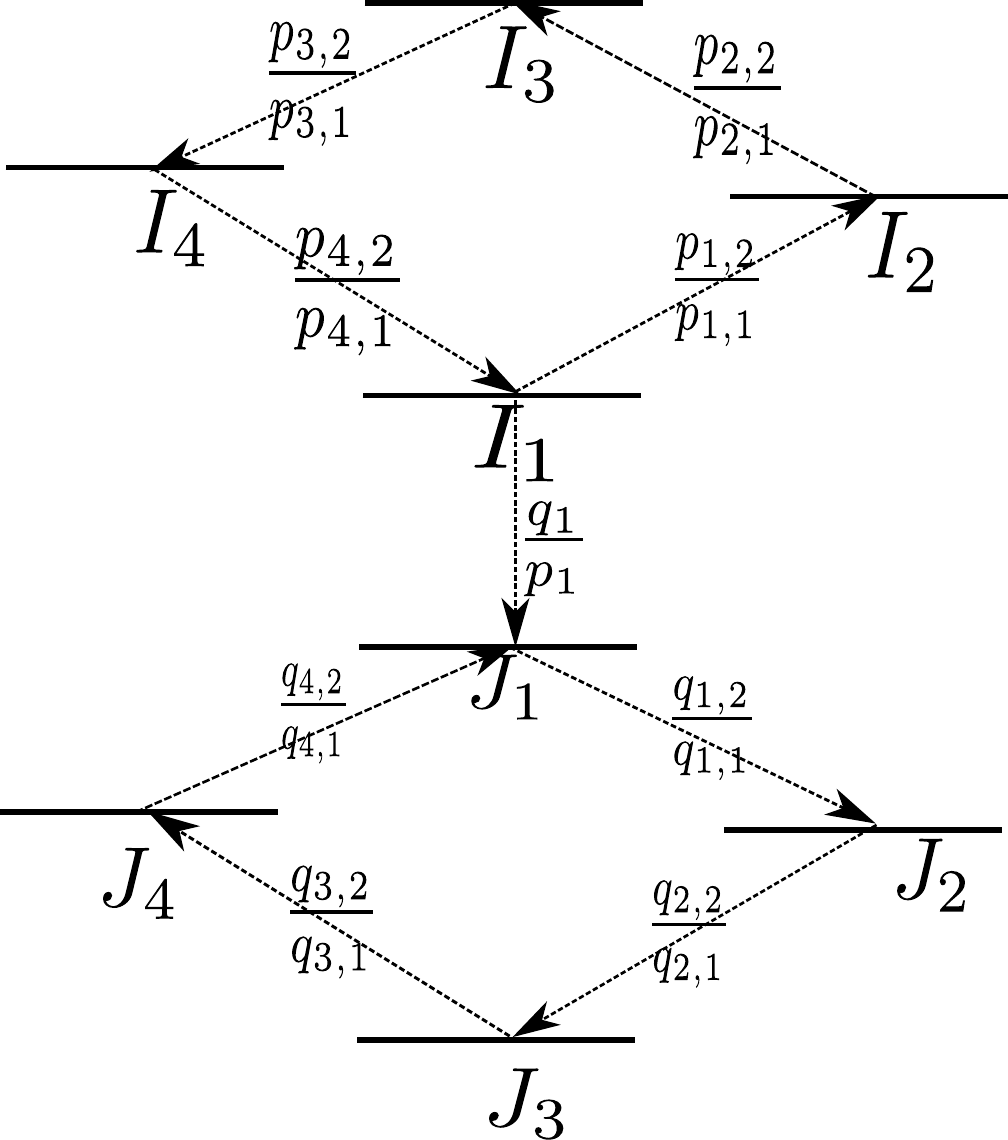}
  \caption{Two cycles of length $k=4$ connected by an edge.  Each dashed line corresponds to a situation of the form described in Figure \ref{fig:intervalpqq} (for all $p \asymp P$).  The frequencies $\alpha_{I_\ell}, \alpha_{J_{\ell}}$ are not depicted here to reduce clutter; however, they will obey the approximate identities \eqref{eq:eq}.}
\label{fig:bicycle}
\end{figure}

Notice that we can interpret each of the relationships in \eqref{eq:eq} as holding $\pmod{p}$ instead of $\pmod{1}$ by multiplying by $p$, thus obtaining the system of equations
\begin{equation}\label{eq:eq2}
\begin{split}
p_{\ell, 1} \alpha_{I_\ell} & \equiv p_{\ell, 2} \alpha_{I_{\ell + 1}} + O \left ( \frac{P}{H} \right ) \mod{p} \\ q_{\ell, 1} \alpha_{J_{\ell}} & \equiv q_{\ell, 2} \alpha_{J_{\ell + 1}} + O \left ( \frac{P}{H} \right ) \mod{p} \\ p_1 \alpha_{I_1} & \equiv q_1 \alpha_{J_1} + O \left ( \frac{P}{H} \right ) \mod{p}
\end{split}
\end{equation}
for all $p \asymp P$.  We can then use the Chinese remainder theorem to replace the $\pmod{p}$ congruences in \eqref{eq:eq2} with $\pmod{Q}$ where $Q \coloneqq \prod_{p \asymp P} p$. A key point for later analysis is that $Q$ is going to be extremely large (of size about $\exp(P) = \exp( X^{\eps \theta})$), so much so that we will eventually be able to drop the congruence $\pmod{Q}$ altogether, once we obtain some more control on the location of the $\alpha_I$.

After applying some algebra to \eqref{eq:eq2} to eliminate all frequencies except $\alpha_{I_1}, \alpha_{J_1}$, we eventually conclude the estimates
\begin{align}
q'_1 \alpha_{I_1} & \equiv O \left ( \frac{P^{k}}{H} \right ) \pmod{Q} \label{cong-1}\\
q'_2 \alpha_{J_1} & \equiv O \left ( \frac{P^{k}}{H} \right ) \pmod{Q} \label{cong-2}\\
p_1 \alpha_{I_1} & \equiv q_1 \alpha_{J_1} + O \left ( \frac{P}{H} \right ) \pmod{Q} \label{cong-3}
  \end{align}
where $q_1' \coloneqq \left|\prod_{\ell = 1}^{k} p_{\ell,1} - \prod_{\ell = 1}^{k} p_{\ell, 2}\right|$ and  $q_2' \coloneqq \left|\prod_{\ell = 1}^{k} q_{\ell,1} - \prod_{\ell = 1}^{k} q_{\ell, 2}\right|$.  The integers $q'_1, q'_2$ are small; in fact the condition \eqref{eq:nearby} will give the bound $q'_1, q'_2 \ll H^{O( \varepsilon)}$.  We can also assume that these integers are non-zero, because the number of intervals $I_{\ell}, J_{\ell}$ and primes $p_{i, j}, q_{i, j}$ for which $q_j'$ could be zero is negligible. It follows then from \eqref{cong-1}, \eqref{cong-2} that
\begin{align*}
  \alpha_{I_1} &\equiv \frac{a_1}{q_1'} Q + \frac{T_{I_1}}{x_{I_1}}  \pmod{Q} \\ 
  \alpha_{J_1} &\equiv \frac{a_2}{q_2'} Q + \frac{T_{J_1}}{x_{J_1}}  \pmod{Q}
\end{align*}
for some $a_1, a_2 \in \mathbb{Z}$, $0 < q_1', q_2' \ll H^{O(\varepsilon)}$, and $T_{I_1}, T_{J_1} \ll X^2 / H^{2-\rho}$, where $x_{I_1}, x_{J_1}$ the starting points of the intervals $I_{1}$, $J_1$, respectively.

Suppose now for simplicity that $q_1' = q_2' = 1$, so that
\begin{align} \label{eq:characterize}
  \alpha_{I_1} &\equiv \frac{T_{I_1}}{x_{I_1}} \pmod{Q} \\ \label{eq:characterize2}
  \alpha_{J_1} &\equiv \frac{T_{J_1}}{x_{J_1}}. \pmod{Q}
\end{align}
Notice that since $I_1 \cap \frac{p_1}{q_1} J_1 \neq \emptyset $ we have $x_{I_1} \approx \frac{p_1}{q_1} x_{J_1}$.  Combining \eqref{eq:characterize}, \eqref{eq:characterize2} with \eqref{cong-3} we obtain the key relationship
$$ T_{I_1} = T_{J_1} + O(PX/H) \pmod{Q};$$
since $T_{I_1}, T_{J_1}$ are much smaller in magnitude than $Q$, we may now drop the congruence and conclude in fact that
$$T_{I_1} = T_{J_1} + O(PX/H);$$
informally speaking, this means that the map $I \mapsto T_I$ is approximately locally constant on the graph ${\mathcal G}$.
Obtaining these quadruples $(I_1, I_2, p_1, p_2)$ with all the described properties is essentially the content of Section~\ref{local-sec}.

A Taylor expansion shows that if $\alpha_{I_1}$ is as in \eqref{eq:characterize}, then $e(-\alpha_{I_1} n) \approx e^{i \theta_{I_1}} n^{2\pi i T_{I_1}}$ with $\theta_{I_1} \in \R$ depending only on $I_1$.  Similarly for \eqref{eq:characterize2}.
Thus there exists a positive proportion set of disjoint intervals $I,J$ connected by an edge in ${\mathcal G}$ such that 
$$
\left | \sum_{n \in I} f(n) n^{2\pi i T_{I}} \right | \gg H \text{ and } \left | \sum_{n \in J} f(n) n^{2\pi i T_{J}} \right | \gg H. 
$$
for some $T_{I}, T_{J} \ll X^2 / H^2$ with $T_{I} = T_{J} + O(PX/H)$.  To proceed further, we claim that the graph ${\mathcal G}$ is essentially an ``expander graph'' and in particular that it has one very large and highly connected component. This is the content of Section \ref{se:global}. 

To see this claim, notice that taking a $O(PX/H)$-spaced set of values $V$ in the range $\{ T: T = O( X^2/H^{2-\rho}) \}$, we can group the intervals $I$ into subsets $\mathcal{A}(V)$ of those intervals $I$ for which $T_{I} = V + O(PX/H)$. Then, because many pairs of intervals $I,J$ connected by an edge in $\mathcal{G}$ belong to the same $\mathcal{A}(V)$, we obtain a large lower bound of the form
\begin{equation} \label{eq:second}
\frac{X}{H} \cdot \left ( \frac{P}{\log P} \right )^2 \ll \sum_{V} \left ( \sum_{p,q \sim P} \sum_{\substack{I \in \mathcal{A}(V) \\ J \in \mathcal{A}(V) \\ \frac{I}{p} \cap \frac{J}{q} \neq \emptyset}} 1 \right )
\end{equation}
where $P \coloneqq H^{\varepsilon}$. 
That is we obtain a lower bound that corresponds to a positive proportion of disjoint intervals $I,J \subset [X, 2X]$ of length $H$ and primes $p,q \asymp P$ such that $\frac{I}{p} \cap \frac{J}{q} \neq \emptyset$. Now, since the exponential sum $\sum_{p \sim H^{\varepsilon}} p^{it}$ exhibits cancellations, we can (using a bit of harmonic analysis) essentially bound the above by
$$
\ll \sum_{V} \left ( \# \mathcal{A}(V)^2 \cdot \frac{H}{X} \cdot \left ( \frac{P}{\log P} \right )^2 \right ) 
$$
Noticing that  $\sum_{V} \#\mathcal{A}(\mathcal{V}) \ll X / H$, we see that the above expression is in turn
\begin{equation} \label{eq:first}
\ll \left ( \sup_{V} \# \mathcal{A}(V) \right ) \cdot \left ( \frac{P}{\log P} \right )^2,
\end{equation}
and therefore, combining \eqref{eq:second} and \eqref{eq:first}, there exists a value $V$ for which $\# \mathcal{A}(V) \gg X / H$. That is, there exists a universal $T \ll X^2 / H^2$ (up to non-essential perturbations by $O(PX/H)$ that we can ignore) such that for a positive proportion of disjoint intervals $I$ of length $H$ we have,
$$
\left | \sum_{n \in I} f(n) n^{i T} \right | \gg H
$$
Averaging over such intervals it follows that,  there exists $T \in \mathbb{R}$ such that $|T| \ll X^2 / H^2$ and
$$
\int_{X}^{2X} \left | \sum_{x < n \leq x + H} f(n) n^{iT} \right | dx \gg X H. 
$$
By the main theorem of \cite{MR} (or rather more precisely its extension to complex valued functions as in \cite[Theorem A.3]{MRT}) this implies that $f$ has to behave essentially as $n^{- iT} \chi(n)$ with $\chi$ a Dirichlet character of bounded conductor and $|T| \ll X^2 / H^2$, thus finishing the proof. 

\subsection{Some final remarks}

It is very likely that it is possible, at the expense of additional technical difficulties, to push our argument down to $H = \exp((\log X)^{1 - \delta})$ for some $\delta > 0$. However we start running into difficulties when $H$ hits $\exp((\log X)^{2/3 + \varepsilon})$ and our argument appears to hit a hard limit when $H$ enters the neighborhood of powers of $\log X$. 


The first obstruction occurs because we require the set of primes $\mathcal{P} \subset [1, H]$ to be sufficiently dense so that at the very least $\prod_{p \in \mathcal{P}} p > X^2$. This implies that $H$ needs to be larger than $\log X$. 

The second obstruction which prevents $H$ from going below $\exp((\log X)^{2/3})$ occurs because we require the exponential sum $\sum_{p \sim H^{\varepsilon}} p^{it}$ to exhibit cancellations for $t$ of size $X$. This is only known for $H > \exp((\log X)^{2/3 + \varepsilon})$ following the work of Vinogradov-Korobov. This obstruction can be circumvented (in the case of the Liouville function, at least) by assuming the Riemann Hypothesis. In that case the exponential sum $\sum_{p \sim H^{\varepsilon}} p^{it}$ will be non-trivially small provided that $H$ is a large power of the logarithm (specifically $H > (\log X)^{3 / \varepsilon}$).  However, we have not verified that the remaining portions of the argument extend to this range (among other things, one would need to make more precise the dependence of various implied constants on the parameter $k$, which now must grow with $X$ instead of being fixed).

\subsection*{Notational conventions.}
As usual $f \ll g$, $g \gg f$ or $f = O(g)$ means that there is an absolute constant $C > 0$ such that $|f| \leq C g$.  If $C$ needs to depend on some parameters then we indicate this by subscripts, for instance $f \ll_\eta g$ denotes the estimate $|f| \leq C_\eta g$ for some $C_\eta$ depending on $g$.  If we write $f = o(g)$ as $X \to \infty$ this means that $|f| \leq c(X) g$ where $c(X)$ is a quantity that goes to zero as $X$ tends to infinity (which may make other quantities dependent on $X$, such as $H$, go to infinity also).  We also write $f \asymp g$ for $f \ll g \ll f$.

We set
$e(x) \coloneqq e^{2\pi ix}$.  The symbol $p$ always denotes a prime, and so do $p', p''$. Given an interval $I = [a,b]$ we define $I/p \coloneqq [a/p, b/p]$. 
Whenever we write
$
\alpha \equiv \beta + O(\eta) \pmod {1}
$
we mean that there exists an absolute constant $C$ such that,
$
\| \alpha - \beta \| \leq C |\eta|
$
where $\| x \|$ denotes the distance of $x$ from the nearest integer. Similarly whenever we write
$
\alpha \equiv \beta + O(\eta) \pmod{q}
$
we mean $\alpha / q \equiv \beta / q + O(\eta / q) \pmod{1}$. 
Given two intervals $I = [a,b]$ and $J = [c,d]$ with $b < c$, whenever we write 
$\text{dist}(I, J) \leq \eta$, we mean that $|c - b| \leq \eta$.  If $I = [a,b]$ and $c>0$, we write $cI \coloneqq [ca,cb]$, thus for instance $I/p = [a/p,b/p]$.

\subsection*{Acknowledgments.} KM was supported by Academy of Finland grant no. 285894.
MR was supported by an NSERC DG grant, the CRC program and a Sloan Fellowship.
TT was supported by a Simons Investigator grant, the James and Carol Collins Chair,
the Mathematical Analysis \& Application Research Fund Endowment, and by NSF grant DMS-1266164.
Part of this paper was written while the authors were in residence at MSRI in Spring 2017,
which is supported by NSF grant DMS-1440140.

\section{Auxiliary results} \label{se:lemma}

We collect here some standard results that will be used (mostly) in section \ref{se:init}. 

In order to use some tools from graph theory, it is convenient\footnote{It should also be possible to work in a purely continuous setting, replacing various summations in our arguments with appropriately normalized integrals, using Fubini's theorem in place of double counting arguments, allowing the intervals under consideration to overlap each other, and with various graph-theoretic inequalities replaced by their continuous counterparts.  We leave the details of this alternate arrangement of the argument to the interested reader.} to replace the continuous integral $\int_X^{2X}\ dx$ in Theorem \ref{main-thm} by something more discrete.  Given $X, H$, define a \emph{$(X,H)$-family of intervals} to be a finite collection $\mathcal{I}$ of intervals $I = [x_I, x_I+H]$ of length $H$ contained in $[X/10, 10X]$, such that any pair of intervals in $\mathcal{I}$ are separated by a distance at least $500H$; in particular, the intervals in $\mathcal{I}$ are disjoint, and thus the cardinality of $\mathcal{I}$ cannot exceed $X/H$.

We then have

\begin{lemma}[Discretizing] \label{disc}
  Let $a(n)$ be a sequence of complex numbers with $|a(n)| \leq 1$ for all integers $n \geq 1$. Let $\eta > 0$ and $X \geq H \geq 1$. Suppose that
  \begin{equation} \label{asmpt}
  \int_{X}^{2X} \sup_{\alpha \in \mathbb{R}} \left | \sum_{x < n \leq x + H} a(n) e(-\alpha n) \right | dx \geq \eta H X.
  \end{equation}
  Then there exist an $(X, H)$-family of intervals $\mathcal{I}$ of cardinality $\geq \frac{\eta X}{1000 H}$ and real numbers $\alpha_I$ associated to each $I \in \mathcal{I}$ such that, for all $I \in \mathcal{I}$,
  \begin{equation} \label{sfan}
  \left | \sum_{n \in I} a(n) e(-\alpha_I n) \right | \geq \frac{\eta H}{2}.
  \end{equation}
\end{lemma}
\begin{proof}
  It follows from \eqref{asmpt} and the pigeonhole principle that there exists $y \in [0, H)$ such that
  \begin{equation} \label{secondeq}
  \sum_{0 \leq \ell < X / H} \left ( \sup_{\alpha \in \mathbb{R}} \left | \sum_{\ell H + y < n \leq (\ell + 1) H + y} a(n) e(-\alpha n) \
	\right| \right ) \geq \eta X
\end{equation}
Given $0 \leq v < 500$, let
$\mathcal{I}_{v}$ be the sub-collection of intervals $I = ((500 \ell + v) H + y, (500 \ell + v + 1) H + y]$ with $\frac{X}{500H} \leq \ell \leq \frac{X}{250 H}$ for which
  $$
  \sup_{\alpha \in \mathbb{R}} \left | \sum_{(500 \ell + v) H + y < n \leq (500 \ell + v + 1) H + y} a(n) e(-\alpha n) \right | \geq \frac{\eta H}{2}. 
  $$
  Let $\mathcal{I} = \bigcup_{0 \leq v < 500} \mathcal{I}_{v}$. 
  It follows from \eqref{secondeq} and the trivial bound $|a(n)| \leq 1$, that
  $$|\mathcal{I}| \cdot H \geq \sum_{I \in \mathcal{I}} \left | \sum_{n \in I} a(n) e(-\alpha n) \right | \geq \frac{\eta X}{2}.$$
  Thus there exists an $0 \leq v < 500$ for which $\mathcal{I}_{v}$ is an $(X, H)$-family of intervals of cardinality $\geq \frac{\eta X}{1000 H}$.  Setting $\mathcal{I} = \mathcal{I}_v$, we obtain the claim.
\end{proof}



The frequency $\alpha_I$ in the above proposition is not unique: one can shift it by any integer, and one can also perturb it by up to a small multiple of $\eta/H$ without significantly affecting \eqref{sfan}.  However, it turns out that modulo these freedoms, there are only a bounded number of choices for $\alpha_I$ (if one views $\eta$ as being fixed).  More precisely, one has

\begin{lemma}[Maximal large sieve]\label{large} Let $H \geq 1$ and let $I$ be an interval of length $10H$. Let $\eta > 0$ be given. Let $|a(n)| \leq 1$ be a sequence of complex numbers. Suppose that there exist $J \geq 1$, frequencies $\alpha_1, \alpha_2, \ldots, \alpha_J \in \mathbb{R}$ and sub-intervals $I_1, I_2, \ldots, I_{J} \subset I$ of length at most $H$ such that
  $$
  \left | \sum_{n \in I_{j}} a(n) e(-\alpha_{j} n) \right | \geq \eta H
  $$
  for all $j = 1, \ldots, J$. Assume $H$ sufficiently large depending on $\eta$.  Then there exist a natural number $K \leq C \eta^{-2}$ with $C$ an absolute constant and frequencies $\beta_1, \ldots, \beta_K$ depending only on $\eta > 0$, the sequence $\{a(\cdot)\}$ and the interval $I$, such that, for each $1 \leq j \leq J$, there exists  $k \in \{1, \dotsc, K\}$ with
  $$
  \| \alpha_j - \beta_k \| \leq \frac{1}{H}
  $$
  where we recall that $\| x \| = \text{dist}(x, \mathbb{Z})$. 
\end{lemma}

\begin{proof}
  Let $\gamma_{1}$ be the frequency $\gamma$ that maximizes the quantity
  \begin{equation} \label{tomax}
  \sup_{L \subset I} \left | \sum_{n \in L} a(n) e(-\gamma n) \right |.
\end{equation}
with the supremum taken over all sub-intervals $L$ of $I$.
For $i \geq 2$ we define $\gamma_{i}$ inductively as the frequency
that maximizes \eqref{tomax} in the region $[0,1] \backslash \bigcup_{j = 1}^{i - 1} [\gamma_j - \frac{1}{H}, \gamma_j + \frac{1}{H}]$. We thus obtain frequencies $\gamma_1, \ldots, \gamma_{R}$ with $R$ a parameter to be chosen later, and moreover $\| \gamma_{i} - \gamma_{j} \| > \frac{1}{H}$ for $i \neq j$.

Using the Carleson-Hunt theorem, it was proven by Montgomery \cite[Theorem 2]{Montgomery2} that one has the maximal large sieve inequality\footnote{At the cost of worsening the dependence on $\eta$ slightly, one could also use the standard large sieve inequality \cite{Montgomery1} here, combined with Lemma \ref{erdosturan} below.}
  $$
  \sum_{r = 1}^{R} \sup_{L \subset I} \left | \sum_{n \in L} a(n) e(-\gamma_{r} n) \right |^2 \leq C (R + H) \sum_{n \in I} |a(n)|^2 
  $$
  with $C$ an absolute constant. The right-hand side is $O(H(R+H))$.  Choosing $R$ to be a large multiple of $\eta^{-2}$, it follows that there are at most $K \ll \eta^{-2}$ frequencies $\gamma_{i}$ for which
  $$
  \sup_{L \subset I} \left | \sum_{n \in I_j} a(n) e(-\gamma_{r} n) \right | \geq \eta H.
  $$
  Therefore for any $\alpha$ lying outside of
  $$
  \bigcup_{i = 1}^{K} \left [\gamma_i - \frac{1}{H}, \gamma_{i} + \frac{1}{H} \right ]
  $$
  we have
  $$
  \sup_{L \subset I} \left | \sum_{n \in L} a(n) e(-\alpha n) \right | < \eta H.
  $$
  Our assumption is that for each $\alpha_j$ with $1 \leq j \leq J$ there exists an interval $I_j$ with $I_j \subset I$ for which
  $$
  \left | \sum_{n \in I_j} a(n) e(-\alpha_j n) \right | \geq \eta H
  $$
  Therefore $\alpha_1, \ldots, \alpha_J \in \bigcup_{i = 1}^{K} [\gamma_{i} - \frac{1}{H}, \gamma_{i} + \frac{1}{H} ]$ and the claim follows. 
\end{proof}

We record also the following variant of the large sieve that we will need in Section \ref{se:global}.
\begin{lemma}[Variant of large sieve]\label{le:largesievish}
  Let $1 \leq H \leq X$ and $R \in \mathbb{N}$. 
  Let $x_1, \dotsc, x_R \in [1, X]$ be $H$-separated (thus $|x_i-x_j| \geq H$ for all $1 \leq i < j \leq R$). Then
  \begin{equation} \label{eq:largesievish}
  \int_{|t| \leq X / H} \left | \sum_{n = 1}^R e(i t \log x_n) \right |^2 dt \ll R \cdot \frac{X}{H}. 
  \end{equation}
\end{lemma}
\begin{proof}
  Let $\Phi(t)$ be a smooth function such that $\Phi(t) \geq 1$ for $|t| \leq 1$ and with $\text{supp } \widehat{\Phi} \subset (-1,1)$. Then the left-hand side of \eqref{eq:largesievish} is
  $$
  \ll \int_{\mathbb{R}} \left | \sum_{n =1}^R e(i t \log x_n) \right |^2 \cdot \Phi \left ( \frac{t H}{X} \right ) dt = \frac{X}{H} \sum_{1 \leq m,n \leq R} \widehat{\Phi} \left ( \frac{X}{H} \log \frac{x_{n}}{x_{m}} \right )  \ll R \cdot \frac{X}{H}
  $$
  as claimed. 
  \end{proof}

We will also need the following tool from harmonic analysis.

\begin{lemma}[Completion of sums]\label{erdosturan}
There exists an absolute constant $\eta_0 > 0$ such that the following holds.
  Let $J$ be an interval of length $H$ and $a(n)$ complex coefficients with $|a(n)| \leq 1$ for all integers $n \geq 1$. Let $I$ be an interval with $I \subset J$. Suppose that $\eta \in (0, \eta_0)$ and $\alpha \in \mathbb{R}$ are such that
  $$
  \left | \sum_{n \in I} a(n) e(-\alpha n) \right | > \eta H.
  $$
  Then there exists $\theta \in \mathbb{R}$ such that $|\theta| \leq \frac{1}{\eta^2 H}$ and
  $$
  \left | \sum_{n \in J} a(n) e(-(\alpha + \theta)n) \right | > \eta^{4} H.
  $$
\end{lemma}
\begin{proof}

  Let $y, z \in \mathbb{R}$ be chosen so that $I = [y, z]$. 
  Let $f$ be a smooth function with $f(n) = 1$ for $n \in I$ , $|f(n)| \leq 1$ for all integers $n$,
  and compactly supported in $[y - \frac{\eta}{100} \cdot H, z + \frac{\eta}{100} \cdot H]$.
  Moreover we can ensure that $f$ is a Schwartz function with $|f^{(j)}(x)| \ll_{j} (\eta H)^{-j}$ for all $x \in \mathbb{R}$ and
  therefore with $|\widehat{f}(x)| = |\int_{\mathbb{R}} f(u) e(-x u) du | \ll_{A} H (1 + \eta H |x|)^{-A}$ for all $A \in \mathbb{N}$. 
  Let 
  $$
  g(\beta) \coloneqq \sum_{n} f(n) e(n \beta).
  $$
  Applying Poisson summation to $g(\beta)$ and using the above bound on $\widehat{f}$ we see that
  \begin{equation}
\label{eq:g2intbound}
  \int_{1 - \frac{1}{\eta^2 H} > |\beta| > \frac{1}{\eta^2 H}} |g(\beta)|^2 d \beta  = \int_{1 - \frac{1}{\eta^2 H} > |\beta| > \frac{1}{\eta^2 H}} \left|\sum_m \widehat{f}(m+\beta)\right|^2 d \beta \ll H \eta^4. 
\end{equation}
Moreover by construction of $g$, 
  $$
  \eta H \leq \left | \sum_{n \in I} a(n) e(-\alpha n) \right | \leq  \left | \int_{\mathbb{T}} \left ( \sum_{n \in J} a(n) e(-(\alpha + \beta) n) \right ) g(\beta) d \beta \right | + \frac{\eta}{100} \cdot H.
  $$

  We split the integral on the right-hand side into two parts, namely $|\beta| \leq \frac{1}{\eta^2 H}$ and the complement. We estimate the part over $|\beta| < \frac{1}{\eta^2 H}$ trivially only using the bound $|g(\beta)| < 2H$. On the second part we apply Cauchy-Schwarz, Plancherel and~\eqref{eq:g2intbound} to see that it is bounded by $\ll \eta^2 H$. 
  Collecting these estimates we conclude that
  $$
  \eta H < 4 H \int_{|\beta| < \frac{1}{\eta^2 H}} \left | \sum_{n \in J} a(n) e(-(\alpha + \beta) n) \right | d \beta
  $$
  Therefore there exists $\beta \in \mathbb{R}$ such that $|\beta| < \frac{1}{\eta^2 H}$ and
  $$
  \left | \sum_{n \in J} a(n) e(-(\alpha + \beta) n) \right | > \eta^4 H 
  $$
  as needed.
\end{proof}

In section \ref{se:init} we will frequently relate the Fourier behavior of $f$ on an interval $I$ with the behavior on dilated intervals $I/p$ for various primes $p$.  The key tool here is

\begin{proposition}[Mean scales down]\label{msd}   Let $x \geq H \geq 1$, and let $f: (x,x+H] \to \C$ obey the bound 
$$ \sum_{n \in (x,  x+H]} |f(n)|^2 \ll H $$
(thus $f = O(1)$ on average on $(x,x+H]$ in an $L^2$ sense).  Then
\begin{equation}\label{first}
\sum_{p \leq H} p \cdot \left |\sum_{m \in (\frac{x}{p}, \frac{x+H}{p}]} f(pm) - \frac{1}{p} \sum_{n \in (x, x+H]} f(n) \right |^2
\ll H^2.
\end{equation}
In particular, by Markov's inequality, for any $\delta > 0$ we have
$$ \sum_{m \in (\frac{x}{p}, \frac{x+H}{p}]} f(pm) = \frac{1}{p} \sum_{n \in (x,x+H]} f(n) + O \left ( \delta \frac{H}{p} \right )$$
for all primes $p \leq H$ outside of an exceptional set ${\mathcal P}$ of primes with $\sum_{p \in {\mathcal P}} \frac{1}{p} \ll \delta^{-2}$.
\end{proposition}

\begin{proof} See 
  \cite[Lemma 4.7]{ElliottBook}.
\end{proof}

We will also need the following number-theoretic estimate, in particular to dispose of some degenerate cases.

\begin{lemma}[Counting nearby products of primes]\label{nte} Let $k \in \mathbb{N}$ and $P', N \geq 3$ be such that $(P')^{k - 1} \gg N$. Write $d = P'^2/(\log P')^2$. Then the number of $2k$-tuples $(p'_{1,1},\dots,p'_{1,k},p'_{2,1},\dots,p'_{2,k})$ of primes in $[P', 2P']$ obeying the condition
\begin{equation} \label{count} \left | \prod_{j=1}^k p'_{2,j} - \prod_{j=1}^k p'_{1,j} \right | \leq C\cdot \frac{(P')^k}{N}
\end{equation}
with $C > 0 $ a constant,
is at most $O_{k,C}( \frac{(P')^{2k}}{N \log^{2k} P'} ) = O_{k,C}( \frac{d^k}{N} )$.

If we also impose the additional condition
\begin{equation} \label{additional} \prod_{j=1}^k p'_{2,j} = \prod_{j=1}^k p'_{1,j} \hbox{ mod } q
\end{equation}
for some modulus $q \in \mathbb{N}$, then the number of tuples is bounded by
$$ O_{k,C} \left ( \frac{d^k}{N} \left ( \frac{1}{\varphi(q)} + \frac{1}{\log N} \right ) \right ).$$
\end{lemma}

\begin{proof}
Since the first claim follows from the second by specializing to $q = 1$ it is enough to prove the second claim. 

First notice that without loss of generality we can assume that $q \leq (\log N)^{3k}$ since otherwise the claim is trivial by replacing products of primes by integers (i.e., using the crude bound that every integer has at most $O_k(1)$ representations as a product of $k$ primes) and counting trivially. 

Let $w$ be a smooth function such that $w(x) = 1$ for $|x| \leq 100 C$. Then,  the number of primes $p'_{1,j}, p'_{2,\ell}$ for which \eqref{count} and \eqref{additional} hold is
\begin{equation}\label{primes}
\ll \sum_{\substack{p'_{1,1}, \ldots, p'_{1, k} \in [P', 2P'] \\ p'_{2,1}, \ldots, p'_{2,k} \in [P', 2P'] \\ p'_{1,1} \ldots p'_{1,k} \equiv p'_{2,1} \ldots p'_{2,k} \mod{q}}} w \left ( N \log \frac{p'_{1,1} \ldots p'_{1,k}}{p'_{2,1} \ldots p'_{2,k}} \right ) 
\end{equation}
Since $q < P'$ and all of the $p'_{1,j}, p'_{2,\ell}$ are primes, we can express the congruence condition using Dirichlet characters, thus
$$ 1_{p'_{1,1} \ldots p'_{1,k} \equiv p'_{2,1} \ldots p'_{2,k} \mod{q}} = 
\frac{1}{\varphi(q)} \sum_{\chi \pmod{q}} \chi(p'_{1,1}) \dots \chi(p'_{1,k}) \overline{\chi(p'_{2,1})} \dots \overline{\chi(p'_{2,k})}$$
where the sum is over all Dirichlet characters of period $q$. Using this identity and the Fourier inversion formula $w(x) = \int_{\mathbb{R}} \widehat{w}(t) e^{2\pi i x t} dt$, we see that the expression \eqref{primes} is equal to
$$
\frac{1}{\varphi(q) N} \sum_{\chi \pmod{q}} \int_{\mathbb{R}} \widehat{w} \left ( \frac{t}{N} \right ) \cdot \left | \sum_{p \in [P', 2P']} p^{i t} \chi(p) \right |^{2k} dt,
$$
Since $q \leq (\log N)^{3k} \ll_k (\log P')^{3k}$, the zero-free region for $L(s, \chi)$ gives 
$$
\sum_{p \in [P', 2P']} p^{it} \chi(p) \ll \frac{P'}{\log P'} \cdot \frac{1}{1 + |t|} \cdot \delta_{\chi = \chi_{0}} + P' \exp( - (\log P')^{1/100});
$$
see for instance \cite[Lemma 2.4]{kou}. Using this pointwise estimate it follows that
$$
\frac{1}{\varphi(q) N} \sum_{\chi \pmod{q}} \int_{|t| \ll \exp((\log P')^{1/100})} \left | \sum_{p \in [P', 2P']} p^{it} \chi(p) \right |^{2k} dt \ll \frac{1}{\varphi(q) N} \frac{P'^{2k}}{(\log P')^{2k}} = \frac{d^{k}}{\varphi(q) N}. 
$$
To bound the part of the integral with large $|t|$ we notice that for arbitrary coefficients $a(n)$, we have the $L^2$ mean value theorem
\begin{equation}\label{l2m}
\int_{\mathbb{R}} \widehat{w} \left ( \frac{t}{N} \right ) \cdot \left | \sum_{A \leq n \leq B} a(n) n^{it} \right |^2 \ll (N + B) \sum_{A \leq n \leq B} |a(n)|^2 
\end{equation}
(see e.g., \cite[Theorem 9.1]{ik}), while from the pointwise estimate we have
\begin{align*}
\frac{1}{\varphi(q) N} \sum_{\chi \pmod{q}} & \int_{|t| \gg \exp((\log P')^{1/100})}  \widehat{w} \left ( \frac{t}{N} \right ) \left | \sum_{p \in [P', 2P']} p^{it} \chi(p) \right |^{2k} dt \\ \ll & P'^2 \exp(- (\log P')^{1/100}) \cdot \sup_{\chi \pmod{q}} \frac{1}{N} \int_{\mathbb{R}} \widehat{w} \left ( \frac{t}{N} \right ) \left | \sum_{p \in [P', 2P']} p^{it} \chi(p) \right |^{2k - 2} dt 
\end{align*}
Since
$$\left | \sum_{p \in [P', 2P']} p^{it} \chi(p) \right |^{2k - 2}  = \left| \sum_{n \in [(P')^{k-1}, (2P')^{k-1}]} a(n) n^{it} \right|^2$$
where
$$ a(n) \coloneqq \chi(n) \sum_{p_1,\dots,p_{k-1} \in [P',2P']: n = p_1 \dots p_{k-1}} 1 = O_k(1)$$
we may thus bound the part of the integral with $|t| > \exp((\log P')^{1/100})$ using \eqref{l2m} by
\begin{align*}
\ll_{k} P'^2 \exp( - (\log P')^{1/100}) \cdot \frac{1}{N} \cdot ( N + P'^{k - 1} ) P'^{k - 1} \ll_{k} \frac{P'^{2k}}{N} \exp( - (\log P')^{1/100})
\end{align*}
as required. Combining the two bounds, the claim follows. 
\end{proof}

\section{Intervals and frequencies} \label{se:init}

Assume we have the hypotheses of Theorem \ref{main-thm}, thus there exists an $\eta > 0$ such that
$$
\int_{X}^{2X} \sup_{\alpha} \left | \sum_{x < n \leq x + H} f(n) e(-\alpha n) \right | d x \geq \eta X H.
$$
Informally speaking, the main purpose of this section is to produce a large set $\mathcal{I}''$ of disjoint intervals $I''$, each of length comparable to some quantity $L$ (which will be slightly shorter than $H$), as well as associated frequencies $\alpha''_{I''}$ with
$$
\left | \sum_{n \in I''} f(n) e(-\alpha''_{I''} n) \right | \gg_{\eta} L,
$$
and a scale $P'$ with the following property: For a positive proportion of quadruples $(I'', J'', p', q') \in \mathcal{I}^2 \times [P', 2P']^2$ with $p', q'$ prime such that $I''$ is close to $\frac{p'}{q'} J''$ we have
$$
\frac{q'}{p''}  \alpha''_{I''} \approx \frac{p'}{p''}  \alpha''_{J''} \pmod{1}
$$
for a positive proportion of primes $p''$ in some range $[P''/2, P'']$ (compare with \eqref{conclusion}). Moreover the ranges $P'', P', L$ are all related by $\log P'' \asymp \log P' \asymp \log L$ and $L \asymp H/P'P''$.
This is the content of Proposition \ref{fan} below. We first need a preliminary proposition.

\begin{proposition}[Scaling down]\label{scale}  Let $1 \leq P \leq Q \leq H \leq X$ and $\eta > 0$, and let $f \colon \mathbb{N} \to \mathbb{C}$ be a $1$-bounded multiplicative function. Assume that $P$ and $\frac{\log Q}{\log P}$ are sufficiently large depending on $\eta$.  Suppose that there exist an $(X,H)$-family $\mathcal{I}$ of intervals of cardinality $\gg_\eta X/H$ and a real number $\alpha_I$ associated to each $I \in \mathcal{I}$ such that
\begin{equation}\label{sumni}
 \left |\sum_{n \in I} f(n) e(-\alpha_I n) \right | \gg_\eta H
\end{equation}
for all $I \in {\mathcal I}$.  Then there exist $P' \in [P, Q/2]$, an $(\frac{X}{P'}, \frac{H}{P'})$-family $\mathcal{I}'$ of intervals of cardinality $\gg_\eta X/H$, and a real number $\alpha'_{I'}$ associated to each $I' \in \mathcal{I}'$, such that
$$ \left |\sum_{n \in I'} f(n) e(-\alpha'_{I'} n) \right | \gg_\eta \frac{H}{P'}$$
for all $I' \in \mathcal{I}'$.  Furthermore, for each $I' \in \mathcal{I}'$, one can find $\gg_\eta \frac{P'}{\log P'}$ pairs $(I, p')$, where $I$ is an interval in $\mathcal{I}$ and $p'$ is a prime in $[P', 2P']$, such that $I/p'$ lies within $3 \frac{H}{P'}$ of $I'$, and such that
$$ p' \alpha_I = \alpha'_{I'} + O_\eta \left ( \frac{P'}{H} \right ) \hbox{ mod } 1.$$
\end{proposition}

The conclusions of Proposition \ref{scale} are depicted schematically in Figure \ref{fig:scaledown}.

\begin{figure} [t]
  \centering
    \scalebox{\scaling}{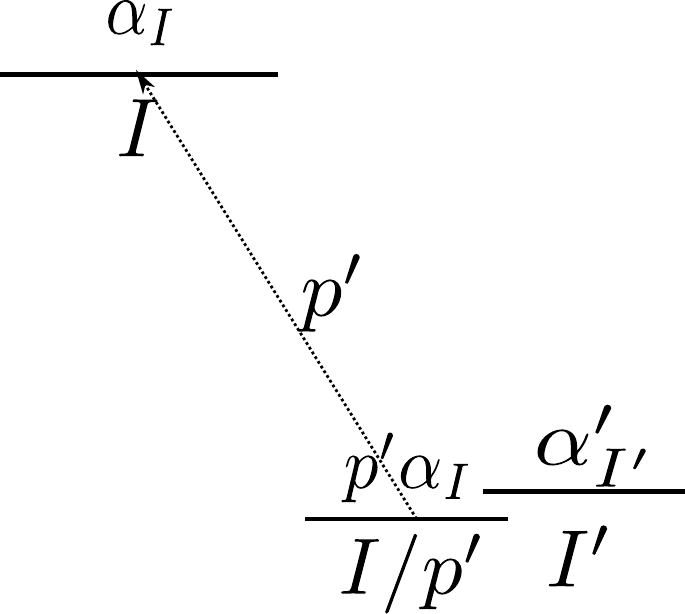}
\caption{A depiction of Proposition \ref{scale}; the frequencies $p' \alpha_I$ and $\alpha'_{I'}$ will be close modulo integers, and each $I' \in {\mathcal I}'$ will be associated to many pairs $(I,p')$ in this fashion.  Compare this with Figure \ref{fig:intervalp}.}
\label{fig:scaledown}
\end{figure}

\begin{proof}  For each $I \in \mathcal{I}$, we apply Proposition \ref{msd} to the function $n \mapsto f(n) e(-\alpha_I n)$ on $I$, and with $\delta$ sufficiently small depending on $\eta$, to conclude that
\begin{equation}\label{fnp}
\left |\sum_{n \in I/p'} f(np') e(-\alpha_I np') \right | \gg_\eta \frac{H}{P'}
\end{equation}
for all primes $p' \in [P,Q]$ outside of an exceptional set ${\mathcal P}_I$ with 
$$ \sum_{p' \in {\mathcal P}_I} \frac{1}{p'} \ll_\eta 1.$$
Summing over all $I \in {\mathcal I}$ (recalling that this collection of intervals has cardinality at most $X/H$), we conclude
$$ \sum_{P \leq p' \leq Q} \frac{1}{p'} \# \{ I \in {\mathcal I}: p' \in {\mathcal P}_I \} \ll_\eta \frac{X}{H}.$$
From Mertens' theorem and the pigeonhole principle, we may thus find $P' \in [P, Q/2]$ such that
$$ \sum_{p' \in [P', 2P']} \# \{ I \in {\mathcal I}: p' \in {\mathcal P}_I \}  \ll_\eta \frac{X}{H \log \frac{\log Q}{\log P}} \frac{P'}{\log P'}.$$
Fix this quantity $P'$.  If $\frac{\log Q}{\log P}$ is large enough, we conclude from the prime number theorem that
$$ \sum_{p' \in [P', 2P']} \# \{ I \in {\mathcal I}: p' \not \in {\mathcal P}_I \}  \gg_\eta \frac{X}{H} \frac{P'}{\log P'},$$
and thus we have \eqref{fnp} for $\gg_\eta \frac{X}{H} \frac{P'}{\log P'}$ pairs $(I,p')$ with $I \in {\mathcal I}$ and $p' \in [P', 2P']$.

As $f$ is multiplicative, we have $f(np') = f(n) f(p')$ unless $n$ is a multiple of $p'$.  The latter contributes at most $O( \frac{H}{p' P} )$ to the left-hand side of \eqref{fnp}, which is negligible compared to the right-hand side as $P$ (and hence $p'$) is large.  Thus we may freely replace $f(np')$ by $f(n) f(p')$, and conclude that
\begin{equation}\label{fnp-2}
\left |\sum_{n \in I/p'} f(n) e(-\alpha_I np') \right | \gg_\eta \frac{H}{P'}
\end{equation}
for $\gg_\eta \frac{X}{H} \frac{P'}{\log P'}$ pairs $(I,p')$.  (Compare with Figure \ref{fig:intervalp}.)  

Let ${\mathcal S}$ denote the collection of these pairs $(I,p')$, and let ${\mathcal I}_1$ denote the collection of all intervals of the form $I/p'$ where $(I, p') \in {\mathcal S}$.  These are intervals in $[0,10X/P']$ of length between $H/2P'$ and $H/P'$.  By a simple greedy algorithm, we may find a subfamily ${\mathcal I}_2$ of these intervals which are separated by distance at least $2 H/P'$, with the property that every interval in ${\mathcal I}_1$ lies within a distance $3 H/P'$ of one of the intervals in ${\mathcal I}_2$.

By \eqref{fnp-2} and Lemma \ref{large}, we can associate to each interval $I' \in {\mathcal I}_2$ some real numbers $\beta_{I',1},\dots,\beta_{I',K(I')}$ for some $K(I') \ll_\eta 1$, with the property that, for each pair $(I,p') \in {\mathcal S}$ with $I/p'$ within $3H/P'$ of $I'$, one has
$$ p' \alpha_I = \beta_{I',k} + O_\eta \left ( \frac{P'}{H} \right ) \hbox{ mod } 1$$
for some $1 \leq k \leq K(I')$.  By adding dummy values of $\beta$ if necessary we may take $K = K(I')$ independent of $I'$.  By the pigeonhole principle, we may find $1 \leq k_0 \leq K$ such that one has
\begin{equation}\label{pai}
p' \alpha_I = \beta_{I',k_0} + O_\eta \left ( \frac{P'}{H} \right ) \hbox{ mod } 1
\end{equation}
for $\gg_\eta \frac{X}{H} \frac{P'}{\log P'}$ triples $(I,p',I')$ with $(I,p') \in {\mathcal S}$ and $I' \in {\mathcal I}_2$ with $\frac{1}{p'} I$ within distance $3 \frac{H}{P'}$ of $I'$.  If we let ${\mathcal T}$ be the collection of such triples, then one can find a subset ${\mathcal I}_3$ of ${\mathcal I}_2$ of cardinality $\gg_\eta \frac{X}{H}$ with the property that for each $I' \in {\mathcal I}_3$, there are $\gg_\eta \frac{P'}{\log P'}$ pairs $(I,p') \in {\mathcal S}$ with $(I,p',I') \in {\mathcal T}$.

For $I' \in {\mathcal I}_3$, pick one of the pairs $(I(I'),p'(I')) \in {\mathcal S}$ with $(I(I'),p'(I'),I') \in {\mathcal T}$, then from \eqref{fnp-2} we have
\begin{equation}\label{fnp-3}
 \left |\sum_{n \in \frac{1}{p'(I')} I(I')} f(n) e(-\alpha_{I(I')} n p'(I')) \right | \gg_\eta \frac{H}{P'}
\end{equation}
while from \eqref{pai} we have
$$ p' \alpha_I = p'(I') \alpha_{I(I')} + O_\eta \left ( \frac{P'}{H} \right ) \hbox{ mod } 1$$
whenever $(I,p') \in {\mathcal S}$ with $(I,p',I') \in {\mathcal T}$.

The interval $I(I')/p'(I')$ lies in $[0,10X/P']$ with length between $H/2P'$ and $H/P'$.  Let $J(I')$ be an interval in $[0,10X/P']$ of length exactly $H/P'$ containing $I(I')/p'(I')$. By Lemma \ref{erdosturan} and \eqref{fnp-3}, we have
$$
\left  |\sum_{n \in J(I')} f(n) e(-\alpha'_{J(I')} n) \right | \gg_\eta \frac{H}{P'}$$
for some real number
$$ \alpha'_{J(I')} = p'(I') \alpha_{I(I')} + O_\eta \left (\frac{P'}{H} \right ).$$
In particular
$$ p' \alpha_I = \alpha'_{J(I')} + O_\eta \left ( \frac{P'}{H} \right ) \hbox{ mod } 1$$
whenever $(I,p') \in {\mathcal S}$ with $(I,p',I') \in {\mathcal T}$.

Setting ${\mathcal I}'$ to be a $500H/P'$-separated collection of $\gg X/H$ intervals of the form $J(I')$ with $I' \in {\mathcal I}_3$, we obtain the claim.
\end{proof}

We are now ready to prove the main result of this section. 

\begin{proposition}\label{fan}   Let $X \geq 2$, $\theta \in (0,1)$, $\eta > 0$, and $\rho \in (0, 1/8)$.
  Let $f: \mathbb{N} \rightarrow \mathbb{C}$ be a multiplicative function with $|f| \leq 1$.
  Suppose that, for $H = X^{\theta}$, we have
  $$
  \int_{X}^{2X} \sup_{\alpha} \left | \sum_{x < n \leq x + H} f(n) e(-\alpha n) \right | d x \geq \eta H X.
  $$
Let $\eps \in (0, \rho / 100)$ be sufficiently small depending on $\theta$ and $\eta$, and assume $X$ is sufficiently large depending on $\theta,\eta$, $\rho$, and $\eps$.  Then there exist $P', P'' \in [X^{\eps^2}, X^\eps]$, an $(\frac{X}{P'P''}, \frac{H}{P'P''})$-family $\mathcal{I}''$ of intervals of cardinality $\gg X/H$, and a real number $\alpha''_{I''}$ associated to each $I'' \in \mathcal{I}''$ such that
$$ \left |\sum_{n \in I''} f(n) e(-\alpha''_{I''} n) \right | \gg_\eta \frac{H}{P' P''}$$
for all $I'' \in \mathcal{I}''$.  Furthermore, there exist $\gg_\eta (\frac{P'}{\log P'})^2 \frac{X}{H}$ quadruples $(I''_1, I''_2, p'_1, p'_2)$ with $I''_1, I''_2$ distinct intervals in $\mathcal{I}''$ and $p'_1,p'_2$ distinct primes in $[P', 2P']$, such that $I''_1$ lies within $50 \frac{H}{P'P''}$ of  $\frac{p'_2}{p'_1} I''_2$, and such that
\begin{equation}\label{pio}
 p'_2 \alpha''_{I''_1} - p'_1 \alpha''_{I''_2} = O_\eta \left ( \frac{(P')^2 P''}{H}\right  ) \hbox{ mod } p''
\end{equation}
for $\gg_\eta \frac{P''}{\log P''}$ primes $p'' \in [P''/2, P'']$.
\end{proposition}

\begin{proof}  By Lemma~\ref{disc}, one can find $(X,H)$-family $\mathcal{I}$ of intervals of cardinality $\gg \eta X/H$ and a real number $\alpha_I$ associated to each $I \in \mathcal{I}$ such that
$$ \left |\sum_{n \in I} f(n) e(-\alpha_I n) \right | \gg \eta H$$
for all $I \in {\mathcal I}$.  Applying Proposition \ref{scale}, one can find $P' \in [X^{\eps^2}, X^\eps]$, an $(\frac{X}{P'}, \frac{H}{P'})$-family $\mathcal{I}'$ of intervals  of cardinality $\gg_\eta X/H$, and a real number $\alpha'_{I'}$ associated to each $I' \in \mathcal{I}'$, such that
$$ \left |\sum_{n \in I'} f(n) e(-\alpha'_{I'} n) \right | \gg_\eta \frac{H}{P'}$$
for all $I' \in \mathcal{I}'$.  Furthermore, for each $I' \in \mathcal{I}'$, one can find $\gg_\eta \frac{P'}{\log P'}$ pairs $(I, p')$, where $I$ is an interval in $\mathcal{I}$ and $p'$ is a prime in $[P', 2P']$, such that $I/p'$ lies within $3 \frac{H}{P'}$ of $I'$ and 
$$ p' \alpha_I = \alpha'_{I'} + O_\eta \left ( \frac{P'}{H} \right ) \hbox{ mod } 1.$$

By a second application of Proposition \ref{scale}, one can find $P'' \in [(X/P')^{\eps^2}, (X/P')^\eps]$, an $(\frac{X}{P' P''}, \frac{H}{P' P''})$-family $\mathcal{I}''$ of intervals of cardinality $\gg_\eta X/H$, and a real number $\alpha''_{I''}$ associated to each $I'' \in \mathcal{I}''$, such that
\begin{equation}\label{hpp}
 \left |\sum_{n \in I''} f(n) e(-\alpha''_{I''} n) \right | \gg_\eta \frac{H}{P' P''}
\end{equation}
for all $I'' \in \mathcal{I}''$.  Furthermore, for each $I'' \in \mathcal{I}''$, one can find $\gg_\eta \frac{P''}{\log P''}$ pairs $(I', p'')$, where $I'$ is an interval in $\mathcal{I}'$ and $p''$ is a prime in $[P''/2, P'']$, such that $I'/p''$ lies within $3 \frac{H}{P' P''}$ of $I''$, and such that
\begin{equation}\label{ppi}
 p'' \alpha'_{I'} = \alpha''_{I''} + O_\eta \left ( \frac{P' P''}{H} \right ) \hbox{ mod } 1.
\end{equation}
Also, since the $I'$ are $500H$-separated, we see that each prime $p''$ is associated to at most one $I'$ in this fashion (for a fixed choice of $I''$).  The above situation is depicted in Figure \ref{fig:scaledown2}.

\begin{figure} [t]
  \centering
    \scalebox{\scaling}{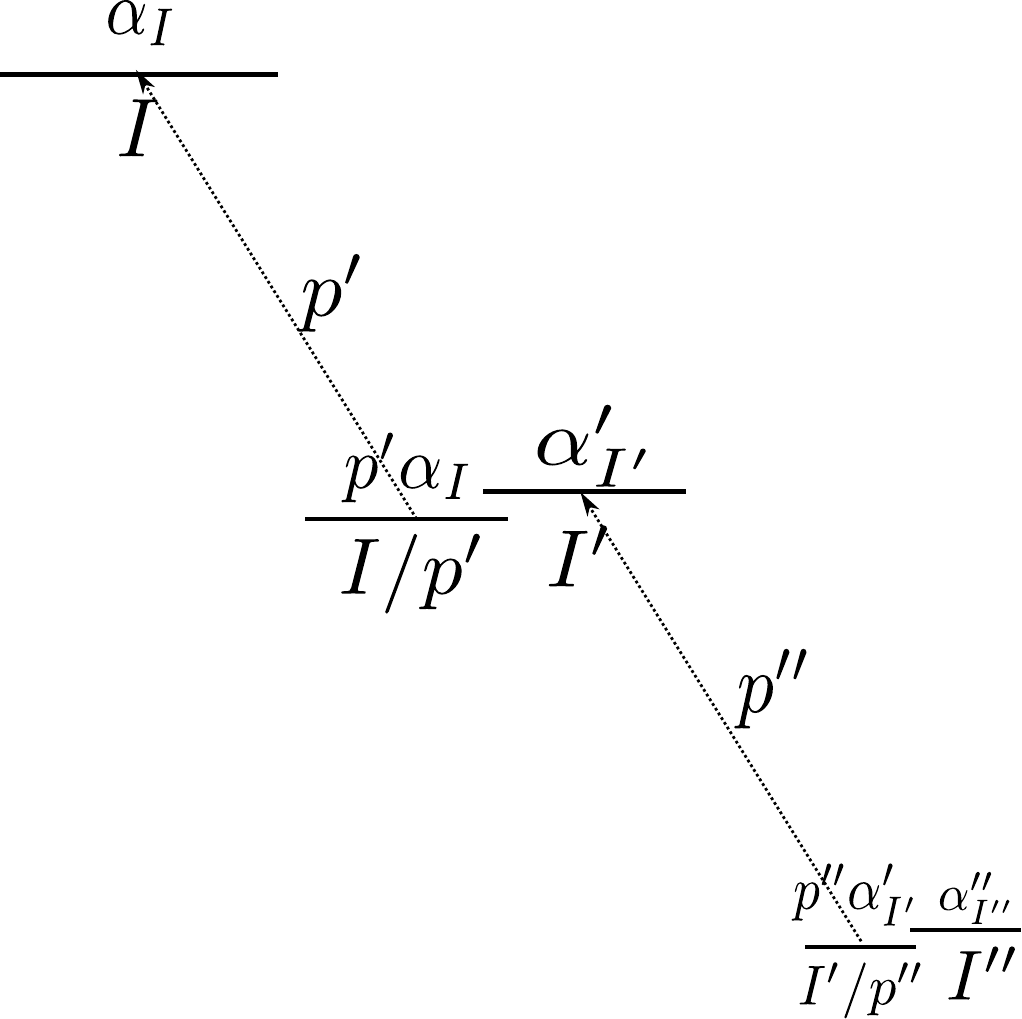}
\caption{The relationship between the intervals $I,I',I''$, primes $p', p''$, and frequencies $\alpha_I, \alpha'_{I'}, \alpha''_{I''}$.}
\label{fig:scaledown2}
\end{figure}

Note that if $I'' \in \mathcal{I}''$, then one can add an arbitrary integer to each real number $\alpha''_{I''}$ without affecting any of the above properties.  In particular, if one adds an integer with an appropriate residue class mod $p''$, one can upgrade \eqref{ppi} to
\begin{equation}\label{ppi-2}
 p'' \alpha'_{I'} = \alpha''_{I''} + O_\eta \left ( \frac{P' P''}{H} \right ) \hbox{ mod } p''
\end{equation}
for any pair $(I',p'')$ appearing previously. By the Chinese remainder theorem, we may thus select $\alpha''_{I''}$ so that \eqref{ppi-2} holds for \emph{all} pairs $(I',p'')$ appearing previously.

Combining the above properties, we see that we can find $\gg_\eta \frac{P'}{\log P'} \frac{P''}{\log P''} \frac{X}{H}$ quintuplets $(I, I', I'', p', p'')$, where $I \in \mathcal{I}$, $I' \in \mathcal{I}'$, $I'' \in \mathcal{I}''$, $p'$ is a prime in $[P', 2P']$, $p''$ is a prime in $[P''/2, P'']$, $\frac{1}{p'} I$ lies within $3 \frac{H}{P'}$ of $I'$, $\frac{1}{p''} I'$ lies within $3 \frac{H}{P'P''}$ of $I''$, and one has the equations
$$ p' \alpha_I = \alpha'_{I'} + O_\eta \left ( \frac{P'}{H} \right ) \hbox{ mod } 1$$
and
$$ p'' \alpha'_{I'} = \alpha''_{I''} + O_\eta \left ( \frac{P' P''}{H} \right ) \hbox{ mod } p''.$$
Multiplying the first equation by $p''$ and combining with the second equation, we conclude in particular that
$$ p' p'' \alpha_I = \alpha''_{I''} + O_\eta \left ( \frac{P' P''}{H} \right ) \hbox{ mod } p''.$$

The number of possible choices for $(I, p'')$ is (trivially) at most $\frac{P''}{\log P''} \frac{X}{H}$.  Applying the Cauchy-Schwarz inequality, we conclude that we can find 
$\gg_\eta (\frac{P'}{\log P'})^2 \frac{P''}{\log P''} \frac{X}{H}$ octuplets $(I, I'_1, I'_2, I''_1, I''_2, p'_1, p'_2, p'')$, where
\begin{itemize}
\item $I \in {\mathcal I}$, $I'_1, I'_2 \in \mathcal {I}'$, $I''_1, I''_2 \in \mathcal{I}''$;
\item $p'_1,p'_2$ are primes in $[P', 2P']$, and $p''$ is a prime in $[P''/2, P'']$;
\item For $i=1,2$, $\frac{1}{p'_i} I$ lies within $3 \frac{H}{P'}$ of $I'_i$, and $\frac{1}{p''} I'_i$ lies within $3 \frac{H}{P'P''}$ of $I''_i$.
\item We have
\begin{equation}\label{p1}
 p'_1 p'' \alpha_I = \alpha''_{I''_1} + O_\eta \left ( \frac{P' P''}{H} \right ) \hbox{ mod } p''
\end{equation}
and
\begin{equation}\label{p2}
 p'_2 p'' \alpha_I = \alpha''_{I''_2} + O_\eta \left ( \frac{P' P''}{H} \right ) \hbox{ mod } p''.
\end{equation}
\end{itemize}

See Figure \ref{fig:scaledown3}.

\begin{figure} [t]
  \centering
    \scalebox{\scaling}{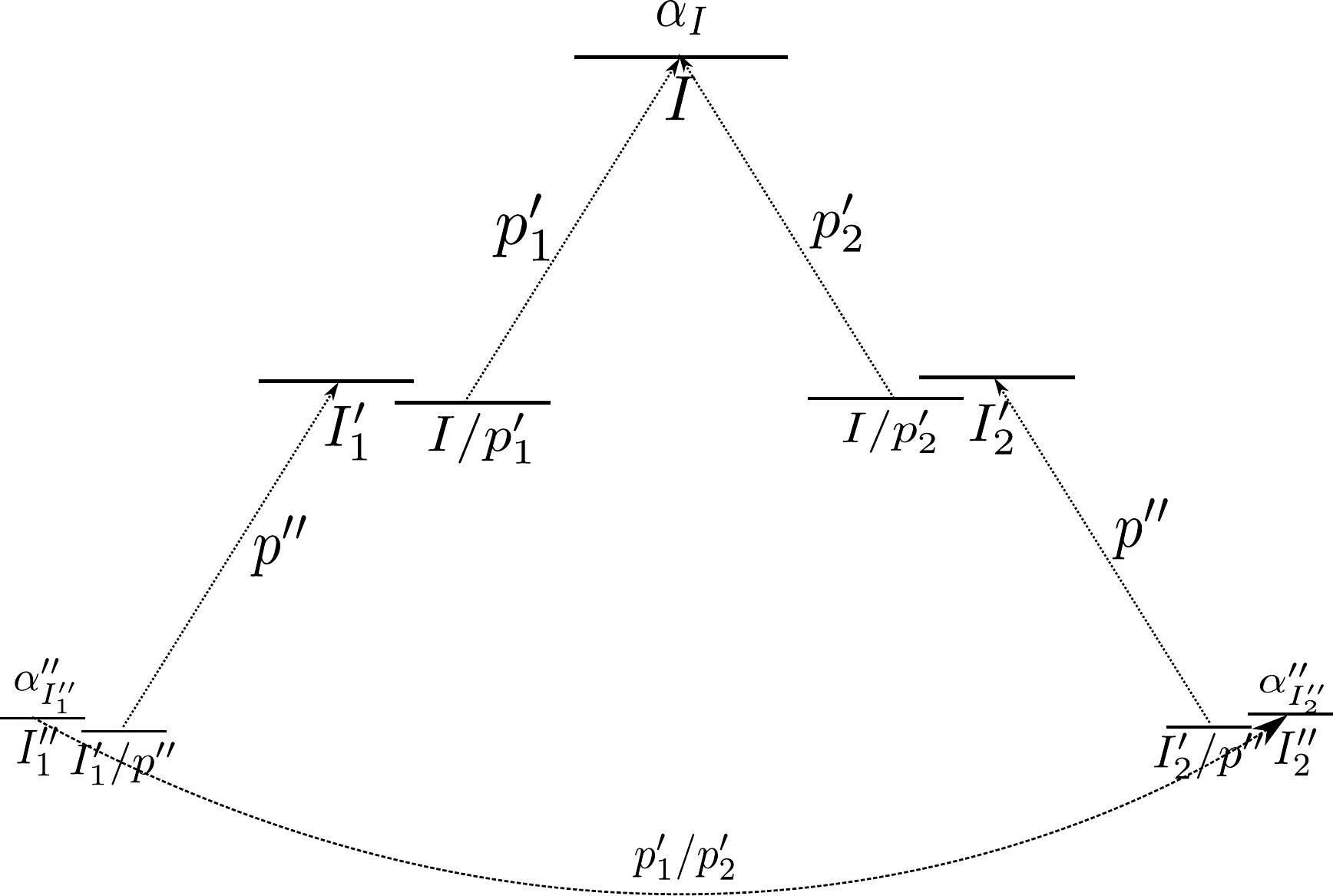}
\caption{The relationship between the intervals $I,I'_1,I'_2,I''_1, I''_2$, primes $p'_1, p''_2, p''$, and frequencies $\alpha_I, \alpha''_{I''_1}, \alpha''_{I''_2}$.  Some frequencies are omitted from the figure to reduce clutter. If this situation occurs for many values of $p''$, we draw a dashed line from $I''_1$ to $I''_2$ labeled by $p'_1/p'_2$.  Compare with Figure \ref{fig:intervalpqq}.}
\label{fig:scaledown3}
\end{figure}

Multiplying \eqref{p1} by $p'_2$ and \eqref{p2} by $p'_1$ and then subtracting, we see that
\begin{equation}\label{p10^4}
 p'_2 \alpha''_{I''_1} - p'_1 \alpha''_{I''_2} = O_\eta \left ( \frac{(P')^2 P''}{H} \right ) \hbox{ mod } p''.
\end{equation}
Also, $p'_1 I'_1$ lies within $6 H$ of $p'_1 p'' I''_1$ and $p'_2 I'_2$ lies within $6 H$ of  $p'_2 p'' I''_2$, so by the triangle inequality $p'_1 p'' I''_1$ and $p'_2 p'' I''_2$ lie at distance at most $24 H$ from each other. Hence, on dividing by $p'_1 p''$, $I''_1$ and $\frac{p'_2}{p'_1} I''_2$ lie at distance at most $48 \frac{H}{P'P''}$ from each other.  In particular, if $p'_1 = p'_2$, then $I''_1 = I''_2$, and similarly $I'_1 = I'_2$.  As a consequence, the number of octuplets with this property is at most $O( \frac{P'}{\log P'} \frac{P''}{\log P''} \frac{X}{H} )$.  Since $P' \geq X^{\eps^2}$ and $X$ is sufficiently large depending on $\eps$, the contribution of this case is thus negligible, so that there are $\gg_\eta (\frac{P'}{\log P'})^2 \frac{P''}{\log P''} \frac{X}{H}$ octuplets $(I, I'_1, I'_2, I''_1, I''_2, p'_1, p'_2, p'')$ with $p'_1 \neq p'_2$.

Observe that if $I''_1, I''_2, p'_1, p'_2$ are fixed, then $I, I'_1, I'_2$ are completely determined by $p''$ thanks to the separation properties of $\mathcal{I}$ and $\mathcal{I}''$; in particular, there are $O( \frac{P''}{\log P''})$ ways to complete the quadruplet $(I''_1, I''_2, p'_1, p'_2)$ to an octuplet.  Similarly, $I''_1$ is completely determined by $I''_2, p'_1, p'_2$ (since there is at most one interval in $\mathcal{I}''$ that lies within $48 \frac{H}{P'P''}$ from $\frac{p'_2}{p'_1} I''_2$).  Thus the number of eligible quadruplets $(I''_1, I''_2, p'_1, p'_2)$  is 
$O( (\frac{P'}{\log P'})^2 \frac{X}{H} )$.  We conclude that there exist $\gg_\eta (\frac{P'}{\log P'})^2 \frac{X}{H}$ quadruplets $(I''_1, I''_2, p'_1, p'_2)$, each of which can be completed to an octuplet in $\gg_\eta \frac{P''}{\log P''}$ ways.  In particular, for such a quadruplet, \eqref{p10^4} holds for $\gg_\eta \frac{P''}{\log P''}$ choices of $p''$ (recalling that $I, I'_1, I'_2$ are completely determined by the remaining coefficients of the octuplet).  The claim follows.
\end{proof}

\section{Local structure of $\alpha''$}\label{local-sec} \label{se:local}

We now analyse the structure of the function $\alpha''$ appearing in Proposition \ref{fan}.  The main result of this section asserts that $\alpha''_{I''}$ locally behaves like $\frac{T}{x_{I''}}$ with $T$ ``not too large'' (and up to a shift $\frac{a}{q}$ with small denominator), where $x_{I''}$ denotes the left endpoint of the interval $I''$.  Crucially, $T$ will not vary much with $I''$, at least ``locally''. It is here that we will rely on the hypothesis $H = X^\theta$ that $H$ is of polynomial size in $X$.

\begin{proposition}\label{local} Let $\theta,\eta,\rho, X, H, f, \eps, P', P'', \mathcal{I}'', \alpha''$ be as in Proposition \ref{fan}.  Then, for $\gg_\eps \frac{X}{H} \left(\frac{P'}{\log P'}\right)^2$ of the pairs $(I''_1, I''_2)$ of intervals in $(\mathcal{I}'')^2$, there exist a natural number
$$1 \leq q \ll H^{\rho},$$
integers $a_1,a_2$, a real number 
$$ T \ll_{\theta,\eta,\eps,\rho} \frac{X^2}{H^{2 - \rho}},$$
and a set ${\mathcal P}(I''_1, I''_2)$ of primes in $[P''/2, P'']$ of cardinality $\gg_{\theta,\eta,\eps,\rho} \frac{P''}{\log P''}$ such that
$$ \alpha''_{I''_j} = \frac{T}{x_{I''_j}} + \frac{a_j}{q} \prod_{p'' \in {\mathcal P}(I''_1, I''_2)}  p'' + O_{\theta,\eta,\eps,\rho} \left ( \frac{1}{H^{1 - \rho}} \right ) \hbox{ mod } \prod_{p'' \in {\mathcal P}(I''_1, I''_2)}  p''$$
for $j=1,2$. Furthermore, for each such pair, there exist primes $p'_1, p'_2 \in [P', 2P']$ such that $I''_1$ lies within $100 \frac{H}{P'P''}$ of  $\frac{p'_2}{p'_1} I''_2$, and such that
\begin{equation}\label{pio-q}
 p'_2 a_1 = p'_1 a_2 \hbox{ mod } q.
\end{equation}
\end{proposition}

\begin{proof} Let $\theta,\eta, \rho, X, H, f, \eps, P', P'', \mathcal{I}'', \alpha''$ be as in Proposition \ref{fan}.  Thus for instance we now have $P'', P' \leq H^{\rho / 100}$. Henceforth we allow implied constants to depend on $\theta,\eta,\eps,\rho$.  We abbreviate
$$ N \coloneqq \# \mathcal{I}'' \asymp \frac{X}{H}$$
for the cardinality of $\mathcal{I}''$ and $d$ for the quantity
$$ d \coloneqq \left (\frac{P'}{\log P'} \right )^2,$$
thus 
the number of quadruples $(I''_1,I''_2,p'_1, p'_2)$ in Proposition \ref{fan} is $\gg d N$.  
We construct a graph ${\mathcal G} = (V,E)$ whose vertices are just the intervals in $\mathcal{I}''$ (thus $V = \mathcal{I}''$ has $N$ vertices), and the edges $e$ are those unordered pairs $e = \{I''_1, I''_2\}$ for which there exist distinct primes $p'_1, p'_2$ in $[P', 2P']$ such that $p_1' I''_1$ lies within $100 \frac{H}{P''}$ of  $p'_2 I''_2$, and such that
\begin{equation}\label{pio-2}
 p'_2 \alpha''_{I''_1} - p'_1 \alpha''_{I''_2} = O_\eta \left ( \frac{(P')^2 P''}{H} \right ) \hbox{ mod } p''
\end{equation}
for a set ${\mathcal P}(e)$ of primes $p''$ in $[P''/2, P'']$ of cardinality $\gg \frac{P''}{\log P''}$ (note that these properties are symmetric in $I''_1$ and $I''_2$).  Observe that the primes $p'_1, p'_2$ are uniquely determined by $I''_1, I''_2$, for if there was another pair of primes $p'_3, p'_4$ with the same properties, then $\frac{p'_2}{p'_1} I''_2$ and $\frac{p'_4}{p'_3} I''_2$ would lie within $200 \frac{H}{P'P''}$ of each other, which implies that
$$ \frac{p'_2}{p'_1} - \frac{p'_4}{p'_3} = O \left ( \frac{H}{X} \right),$$
but if $(p'_1,p'_2) \neq (p'_3,p'_4)$ then the left-hand side has magnitude at least $\frac{1}{p'_3 p'_1} \gg X^{-2\eps^2}$, which leads to a contradiction if $\eps$ is small enough and $X$ is large enough.  Thus, by Proposition \ref{fan}, we see that the number of edges in ${\mathcal G}$ is $\gg dN$.  On the other hand, the degree of each vertex in ${\mathcal G}$ is $O(d)$, since for fixed $I''_1$ there are only $O(d)$ choices for $p'_1$ and $p'_2$, and $I''_2$ is uniquely determined by these choices.  Thus ${\mathcal G}$ has $\asymp dN$ edges and the mean degree of ${\mathcal G}$ is $\asymp d$.

At present, the sets ${\mathcal P}(e)$ of primes associated to each edge $e$ are large, but the intersections ${\mathcal P}(e_1) \cap \dots \cap {\mathcal P}(e_k)$ could be small.  This will cause difficulties later.  To get around this problem we use a random refinement trick of Gowers \cite{gowers-4-aps}.
Let $\mathbf{p}''$ be a prime in $[P''/2,P'']$ selected uniformly at random, and let $\mathbf{G} = (V, \mathbf{E})$ be the subgraph of ${\mathcal G}$ consisting of the same vertex set $V$ as ${\mathcal G}$, and with the edge set $\mathbf{E}$ consisting of all edges $e \in E$ with $\mathcal{P}(e)$ containing $\mathbf{p}$.  By the prime number theorem, each edge has probability $\gg 1$ of lying in $\mathbf{G}$, so by linearity of expectation the expected number of edges in ${\mathbf G}$ is $\gg dN$. In particular, we see that with probability $\gg 1$, the random graph $\mathbf{G}$ has $\gg dN$ edges.  Of course, $\mathbf{G}$ has maximum degree $O(d)$ since it is a subgraph of ${\mathcal G}$.  As we shall see later, the advantage of working with $\mathbf{G}$ instead of ${\mathcal G}$ is that the intersections ${\mathcal P}(e_1) \cap \dots \cap {\mathcal P}(e_k)$ have a high probability of being large when $e_1,\dots,e_k$ are all constrained to lie in ${\mathcal G}$.

If $\mathbf{A}$ is the adjacency matrix of $\mathbf{G}$, then by the preceding discussion we have $1^T \mathbf{A} 1 \gg d N$ (where $1$ denotes the all-ones column vector) with probability $\gg 1$.  By the Blakley-Roy inequality \cite{blakely}, we now see that for any natural number $k$, we have
$1^T \mathbf{A}^{k} 1 \gg_{k} d^{k} N$ with probability $\gg 1$.  
That is to say, with probability $\gg 1$, the number of  $(k+1)$-tuples $(I''_0,\dots,I''_k)$ in $V^{k+1}$ such that $\{I''_j, I''_{j+1}\} \in \mathbf{E}$ for $j=0,\dots,k-1$ is $\gg_{k} d^{k} N$. 

Now let $k$ be the first even integer for which
$$ d^k \geq N^{2} \cdot d.$$
Then (since $P', P'' \leq X^\eps$) we have $k = O(1)$ and
\begin{equation}\label{dke}
N^{2} d \leq d^k \leq N^{2} d^3.
\end{equation}
In particular, we may allow implied constants to depend\footnote{If one were to extend the arguments here to smaller values of $H$, one would need to pay more attention as to the precise dependence of these constants on $k$.} on $k$.
From the preceding discussion, with probability $\gg 1$, the number of $(k+2)$-tuples
\begin{equation}\label{tuple}
 (I''_{k/2,1}, \dots,I''_{0,1}, I''_{0,2}, \dots, I''_{k/2,2}) \in V^{k+2}
\end{equation}
such that $\{ I''_{j,1}, I''_{j+1,1} \}, \{ I''_{j,2}, I''_{j+1,2} \}, \{ I''_{0,1}, I''_{0,2} \} \in \mathbf{E}$ for $j=0,\dots,k/2-1$ is $\gg d^{k+1} N$.  This situation is depicted in Figure \ref{fig:precauchy}.

\begin{figure} [t]
  \centering
  \scalebox{\scaling}{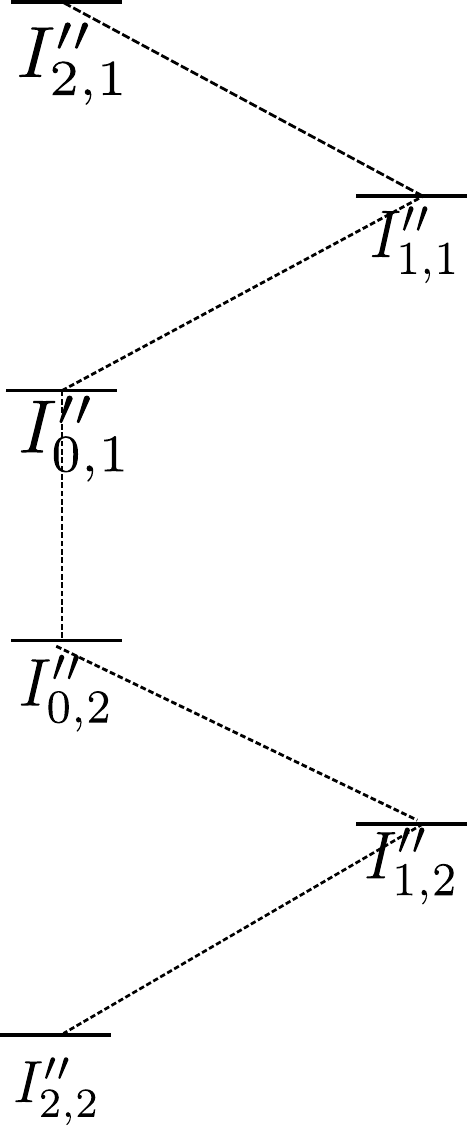}
\caption{A tuple \eqref{tuple} in $\mathbf{G}$ with $k=4$.  Here we discard any orientation or labeling of the edges of $\mathbf{G}$.}
\label{fig:precauchy}
\end{figure}

The number of possible choices for the quadruplet $(I''_{k/2,1}, I''_{0,1}, I''_{0,2}, I''_{k/2,2})$ is $O( dN^3)$, since there are $N^3$ choices for $I''_{k/2,1}, I''_{0,1}, I''_{k/2,2}$, and once $I''_{0,1}$ is fixed there are $O(d)$ choices for $I''_{0,2}$.  Thus by the Cauchy-Schwarz inequality, with probability $\gg 1$, we have there are $\gg (d^{k+1} N)^2 / (dN^3) = d^{2k+1}/N$ pairs of such tuples with a common quadruplet $(I''_{k/2,1}, I''_{0,1}, I''_{0,2}, I''_{k/2,2})$.  Relabeling, we conclude\footnote{This bound also follows from the work of Sidorenko \cite{sidorenko}, as the graph consisting of two $k$-cycles (with $k$ even) connected by an edge is one of the confirmed cases of Sidorenko's conjecture.} that with probability $\gg 1$, the number of $2k$-tuples
\begin{equation}\label{tuple-2}
 \vec I'' \coloneqq (I''_{j,i})_{j \in \{0,1,\ldots, k -1 \}; i = 1,2} \in V^{2k}
\end{equation}
such that $\{ I''_{j,i}, I''_{j+1,i}\}, \{ I''_{0,1}, I''_{0,2} \} \in \mathbf{E}$ for $j=0,\dots,k-1$, $i=1,2$ is $\gg d^{2k+1}/N$, where we adopt the periodic convention $I''_{k,i} = I''_{0,i}$ for $i=1,2$.  In particular, by definition of $\mathbf{G}$, we have
$$\mathbf{p} \in \mathcal{P}( \{ I''_{0,1}, I''_{0,2} \}  ) \text{  and  } \mathbf{p} \in \mathcal{P}( \{I''_{j,i}, I''_{j+1,i}\} )$$
for all $j = 0, 1, \ldots, k - 1$ and $i=1,2$.  The situation is depicted in Figure \ref{fig:postcauchy}.

\begin{figure} [t]
  \centering
    \scalebox{\scaling}{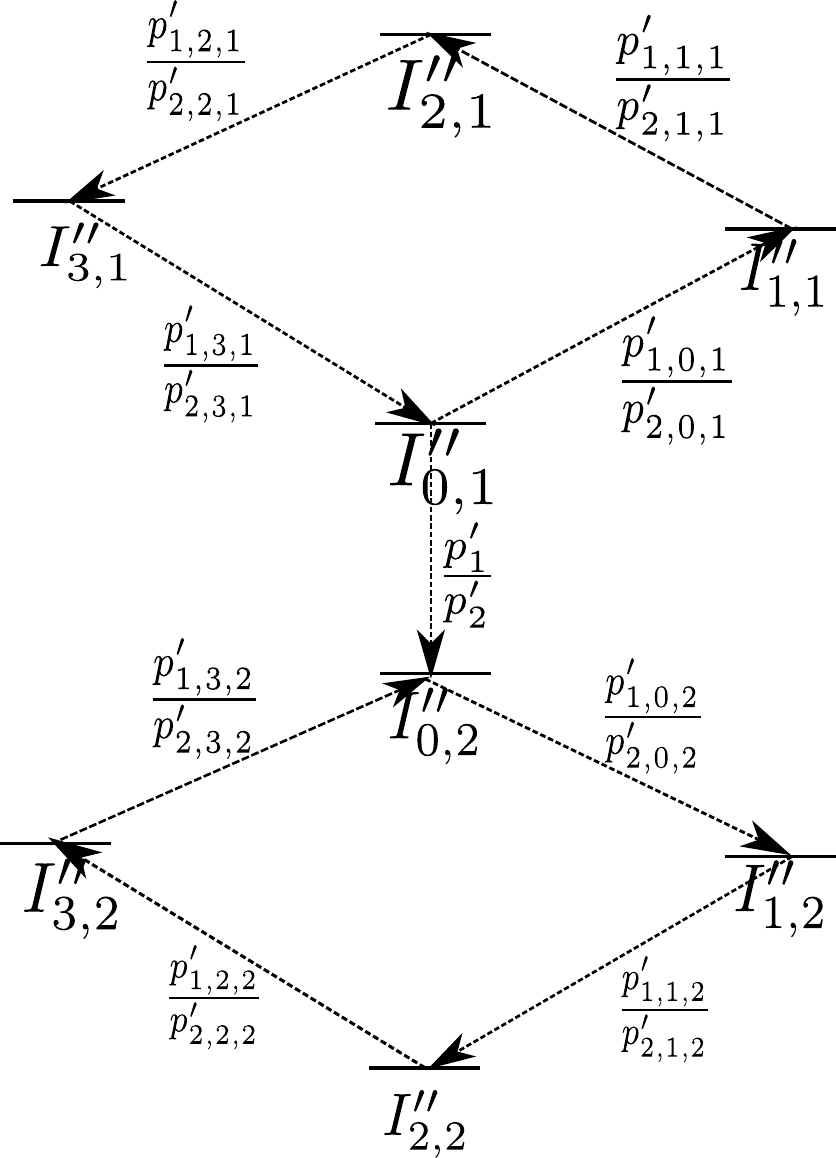}
\caption{A tuple \eqref{tuple-2} in $\mathbf{G}$ with $k=4$, with the orientation and labels restored.  Compare with Figure \ref{fig:bicycle}.}
\label{fig:postcauchy}
\end{figure}

Call the $2k$-tuples $\vec I''$ of the above form \emph{good}, thus there are $\gg d^{2k+1}/N$ good tuples.  Given a good tuple, to each edge $\{I''_{j,i}, I''_{j+1,i}\}$ we have (uniquely determined) primes $p'_{1,j,i}, p'_{2,j,i}$ in $[P', 2P']$, such that
$I''_{j+1,i}$ lies within $100 \frac{H}{P'P''}$ of $\frac{p'_{1,j,i}}{p'_{2,j,i}} I''_{j,i}$ for $j= 0,1,\ldots, k - 1$ and $i=1,2$; we also have primes $p'_1,p'_2 \in [P', 2P']$ such that $I''_{0,2}$ lies within $100 \frac{H}{P'P''}$ of $\frac{p'_1}{p'_2} I''_{0,1}$.  Again, we refer the reader to Figure \ref{fig:postcauchy} for a depiction of these relationships.  Iterating the former claim, we see that $I''_{0,i}$ lies within $O( \frac{H}{P'P''})$ from $\frac{\prod_{j=1}^k p'_{2,j,i}}{\prod_{j=1}^k p'_{1,j,i}} I''_{0,i}$ for $i=1,2$, thus
$$ \frac{\prod_{j=1}^k p'_{2,j,i}}{\prod_{j=1}^k p'_{1,j,i}} = 1 + O \left ( \frac{H}{X} \right ) = 1 + O_{\eps, k} \left ( \frac{1}{N} \right ).$$
Multiplying out, we conclude that
\begin{equation}\label{qb}
\begin{split}
\prod_{j=1}^k p'_{2,j,i} - \prod_{j=1}^k p'_{1,j,i} &\ll \frac{(P')^k}{N} \ll_{\eps} \frac{d^{k/2} (\log P')^k}{N}  
\ll_{\eps} d^{2}
\end{split}
\end{equation}
thanks to \eqref{dke}.

We now eliminate some degenerate cases.  Suppose $\prod_{j=1}^k p'_{2,j,1} - \prod_{j=1}^k p'_{1,j,1} = 0$. Then, by the fundamental theorem of arithmetic, the $p'_{1,j,1}$ are a permutation of the $p'_{2,j,1}$.  By the prime number theorem, the total number of possibilities for the $p'_{1,j,1}, p'_{2,j,1}$ is then at most $\ll_k (P' / \log P')^k \ll d^{k/2}$.  By Lemma \ref{nte}, there are $O( d^{k} / N )$ choices for $p'_{1,j,2}, p'_{2,j,2}$, and finally there are $O(d)$ possibilities for $p'_1,p'_2$ and $O(N)$ possibilities for $I''_{0,1}$.  All the other $I''_{j,i}$ are uniquely determined by this data, so the number of tuples with $\prod_{j=1}^k p'_{2,j,1} - \prod_{j=1}^k p'_{1,j,1} = 0$ is
$$ \ll d^{k/2} \frac{d^k}{N} d N  = d^{3k/2 + 1}$$
which is negligible compared to $d^{2k+1}/N$ thanks to \eqref{dke}.  Thus there are $\gg d^{2k+1}/N$ good tuples for which
$\prod_{j=1}^k p'_{2,j,1} - \prod_{j=1}^k p'_{1,j,1}$ does not vanish.  Repeating this argument for $\prod_{j=1}^k p'_{2,j,2} - \prod_{j=1}^k p'_{1,j,2}$, we may see that with probability $\gg 1$, there are $\gg d^{2k+1}/N$ good tuples for which $\prod_{j=1}^k p'_{2,j,i} - \prod_{j=1}^k p'_{1,j,i} \neq 0$ for $i=1,2$. We will call such good tuples \emph{non-degenerate}.

Another case we would like to exclude is when the set 
$$ \mathcal{P}(\vec I'') \coloneqq \bigcap_{j=1}^k \bigcap_{i=1,2} \mathcal{P}(  \{I''_{j,i}, I''_{j+1,i}\} ) \cap \mathcal{P}( \{ I''_{0,1}, I''_{0,2} \})$$ 
is unusually small, say
\begin{equation}\label{jk}
 \# \mathcal{P}(\vec I'') \leq \delta \frac{P'}{\log P'}
\end{equation}
for some small $\delta>0$ depending on $\eps,\theta,\rho,\eta$) to be chosen later.
Define a \emph{candidate tuple} to be a tuple $\vec I'' = (I''_{j,i})_{j \in \{0,1,\ldots, k - 1\}; i = 1,2} \in V^{2k}$ with $\{I''_{0,1}, I''_{0,2}\} \in E$, $\{I''_{j,i}, I''_{j+1,i}\} \in E$ for $j=0,\dots,k-1$, and $i = 1,2$ obeying \eqref{jk} and with $\prod_{j=1}^k p'_{2,j,i} - \prod_{j=1}^k p'_{1,j,i}$ non-vanishing for $i=1,2$.  Observe that a tuple $\vec I''$ is a non-degenerate good tuple obeying \eqref{jk} precisely if it is a candidate tuple with $\mathbf{p} \in \mathcal{P}(\vec I'')$.  In particular, the probability that a given candidate tuple is actually good is $O(\delta )$.  On the other hand, from two applications of Lemma \ref{nte}, the number of candidate tuples is at most
$$ \ll_\eps N \times d \times (\frac{d^k}{N})^2 = d^{2k+1}/N,$$
and so, by linearity of expectation, the expected number of good tuples obeying \eqref{jk} is $O_( \delta d^{2k+1}/N)$.
On the other hand, with probability $\gg 1$ we have $\gg d^{2k+1}/N$ non-degenerate good tuples.  With $X$ large enough (which makes $P'$ large compared with $\eta,\eps,\rho,\theta$), and setting $\delta$ sufficiently small depending on $\eta,\eps,\rho,\theta$, we thus have with positive probability that there are $\gg d^{2k + 1} / N$ non-degenerate good tuples $\vec I''$ for which
\begin{equation}\label{vg}
 \# \mathcal{P}( \vec I'' ) > \delta \frac{P'}{\log P'}.
\end{equation}
Let us call such tuples \emph{very good}, thus we can find a \emph{deterministic} choice of $\mathbf{p}$ such that there are $\gg d^{2k+1}/N$ very good tuples.  

Henceforth $\mathbf{p}$ is chosen deterministically as above.
Let $\vec I''$ be a very good tuple, with attendant primes $p'_{1,j,i}, p'_{2,j,i}$ and $p'_1,p'_2$ for $j \in \{0,1,\ldots, k - 1\}$ and $i=1,2$.  From \eqref{pio-2}, \eqref{vg} we see that there is a collection ${\mathcal P}(\vec I'')$ of primes in $[P''/2,P'']$ of cardinality
$$ \# {\mathcal P}(\vec I'') \gg \frac{P'}{\log P'}$$
such that
$$ p'_{2,j,i} \alpha''_{I''_{j,i}} - p'_{1,j,i} \alpha''_{I''_{j+1,i}} = O \left ( \frac{(P')^2 P''}{H} \right ) \hbox{ mod } p'' $$
and
$$ p'_{2} \alpha''_{I''_{0,1}} - p'_{1} \alpha''_{I''_{0,2}} = O_\eta \left ( \frac{(P')^2 P''}{H} \right ) \hbox{ mod } p'' $$
for all $p'' \in {\mathcal P}(\vec I'')$, $j \in \{0,1,\ldots, k-1\}$, and $i=1,2$.  For $X$ large enough, the error term $O_\eta( \frac{(P')^2 P''}{H} )$ is less than $1/2$ in magnitude; thus the nearest integer to $p'_{2,j,i} \alpha''_{I''_{j,i}} - p'_{1,j,i} \alpha''_{I''_{j+1,i}}$ is divisible by all the primes in ${\mathcal P}(\vec I'')$, and is hence divisible by the product $Q \coloneqq \prod_{p'' \in {\mathcal P}(\vec I'')}  p''$ of all the primes.  Thus
$$ p'_{2,j,i} \alpha''_{I''_{j,i}} - p'_{1,j,i} \alpha''_{I''_{j+1,i}} = O \left ( \frac{(P')^2 P''}{H} \right ) \hbox{ mod } Q$$
for all $j = 0,1, \ldots, k - 1$ and $i = 1, 2$ and similarly
\begin{equation}\label{pew-pew}
 p'_{2} \alpha''_{I''_{0,1}} - p'_{1} \alpha''_{I''_{0,2}} = O \left ( \frac{(P')^2 P''}{H} \right ) \hbox{ mod } Q.
\end{equation}
We multiply the former equation by $\prod_{0 \leq j'<j} p'_{1,j',i} \prod_{j < j' < k} p'_{2,j',i}$  and sum the telescoping series for $j=0,\dots,k-1$ to conclude that
$$ \left (\prod_{j=0}^{k-1} p'_{2,j,i} \right ) \alpha''_{I''_{0,i}} - \left (\prod_{j=0}^{k-1} p'_{1,j,i} \right ) \alpha''_{I''_{0,i}} = O \left ( \frac{(P')^{k+1} P''}{H} \right ) \hbox{ mod } Q.$$
This implies that
\begin{equation}\label{qi}
 q_i \alpha''_{I''_{0,i}}  = O \left ( \frac{(P')^{k+1} P''}{H} \right ) \hbox{ mod } Q
\end{equation}
for $i=1,2$,
where $q_i$ is the non-negative integer
$$ q_i \coloneqq \left |\prod_{j=1}^k p'_{2,j,i} - \prod_{j=1}^k p'_{1,j,i} \right |.$$
As $\vec I''$ is non-degenerate, $q_i$ is strictly positive.  From \eqref{qb} we conclude that
$$ 1 \leq q_i \ll d^{2}.$$
From \eqref{qi}, we may write
\begin{equation}\label{aoi}
 \alpha''_{I''_{0,i}} = \frac{b_i}{q_i} Q + O \left ( \frac{(P')^{k+1} P''}{H} \right ) \hbox{ mod } Q
\end{equation}
for $i=1,2$ and some integers $b_1,b_2$.  Inserting this into \eqref{pew-pew}, we conclude that
$$ \left (p'_2 \frac{b_1}{q_1} - p'_1 \frac{b_2}{q_2} \right ) Q = O \left ( \frac{(P')^{k+2} P''}{H}\right  ) \hbox{ mod } Q$$
or equivalently
$$ p'_2 \frac{b_1}{q_1} - p'_1 \frac{b_2}{q_2} = O \left ( \frac{(P')^{k+2} P''}{Q H} \right ) \hbox{ mod } 1.$$
The left-hand side is a rational of denominator at most $O(d^4)$.  Meanwhile, since ${\mathcal P}(\vec I'')$ has cardinality $\gg \frac{P'}{\log P'} \gg X^{\eps^2} / \log X$, we have
\begin{equation}\label{ql}
 Q \gg \exp( c X^{\eps^2} )
\end{equation}
for some $c > 0$ depending on $\eps,\rho,\theta,\eta$.   Thus the expression $O( \frac{(P')^{k+2} P''}{Q H} )$ is far smaller than the denominator on the left-hand side, and hence
$$ p'_2 \frac{b_1}{q_1} - p'_1 \frac{b_2}{q_2} = 0 \hbox{ mod } 1.$$
Since we can modify $\frac{b_1}{q_1}$ and $\frac{b_2}{q_2}$ by arbitrary integers without affecting the claimed properties, and $p'_1,p'_2$ are distinct, we may in fact assume without loss of generality that
$$ p'_2 \frac{b_1}{q_1} - p'_1 \frac{b_2}{q_2} = 0,$$
thus we can write $\frac{b_i}{q_i} = \frac{ap'_i}{q}$ for some integer $a$, some $1 \leq q \ll d^{2}$, and for $i=1,2$.  In particular, from \eqref{aoi} we have
$$ \alpha''_{I''_{0,i}} = \frac{ap'_i}{q} Q + O \left ( \frac{(P')^{k+1} P''}{H} \right ) \hbox{ mod } Q$$
for $i=1,2$; from \eqref{dke} we thus have
$$ \alpha''_{I''_{0,i}} = \frac{ap'_i}{q} Q + O \left ( d^3 P'' \cdot \frac{X}{H^2} \right ) \hbox{ mod } Q.$$
We can then write 
$$ \alpha''_{I''_{0,1}} = \frac{ap'_1}{q} Q + \frac{T}{x_{I''_{0,1}}} \hbox{ mod } Q$$
for some real number 
\begin{equation}\label{tete}
T = O \left ( d^{3} P'' \cdot \frac{X^2}{H^2} \right ),
\end{equation}
and we then write
$$ \alpha''_{I''_{0,2}} = \frac{ap'_2}{q} Q + \frac{T}{x_{I''_{0,2}}} + \theta \hbox{ mod } Q$$
for some real number 
\begin{equation}\label{thet}
\theta = O \left ( d^{3} P'' \cdot \frac{X}{H^2} \right ).
\end{equation}
Inserting these equations back into \eqref{pew-pew}, we obtain
$$ T \left (\frac{p'_2}{x_{I''_{0,1}}} - \frac{p'_1}{x_{I''_{0,2}}} \right ) - p'_1 \theta = O \left ( \frac{(P')^2 P''}{H} \right ) \hbox{ mod } Q.$$
Since $I''_{0,2}$ lies within $100 \frac{H}{P'P''}$ of $\frac{p'_1}{p'_2} I''_{0,1}$, we have
$$ x_{I''_{0,2}} = \frac{p'_1}{p'_2} x_{I''_{0,1}} + O \left ( \frac{H}{P' P''} \right )$$
and hence by \eqref{tete}
$$ p'_1 \theta = O \left ( \frac{d^3}{H} + \frac{(P')^2 P''}{H} \right ) \hbox{ mod } Q.$$
Combining this with \eqref{thet}, \eqref{ql} we conclude that
$$ \theta = O \left ( \frac{d^3}{H} + \frac{(P')^2 P''}{H} \right ) $$
and thus
$$ \alpha''_{I''_{0,i}} = \frac{ap_i'}{q} Q + \frac{T}{x_{I''_{0,i}}} + O \left ( \frac{d^3}{H} + \frac{(P')^2 P''}{H} \right )  \hbox{ mod } Q$$
for $i=1,2$.  

Finally, by two applications of Lemma \ref{nte}, each pair $(I''_{0,1}, I''_{0,2})$ is associated to at most $(O( \frac{d^k}{N} ))^2$ very good tuples; since there are $\gg d^{2k+1}/N$ such tuples, the number of pairs $(I''_{0,1}, I''_{0,2})$ that arise in this fashion is
$$ \gg \frac{d^{2k+1}/N}{( \frac{d^k}{N} )^2} \gg dN \gg \frac{X}{H} \left (\frac{P'}{\log P'}\right )^2.$$
The claim follows.
\end{proof}

\section{Global structure of $\alpha''$} \label{se:global}

Proposition \ref{local} gives some control on $\alpha''$, but it is currently ``local'' because the parameters $T,q$ that arise in this control depend on the pair $I''_1,I''_2$.  Fortunately, one can use the ``mixing'' or ``ergodicity'' properties of the graph of such pairs to convert this local control to global control.  To do this we first need a lemma.

\begin{lemma}[Mixing lemma]\label{mixlem} Let $\theta,\eta,X, H, f, \rho,\eps, P', P'', \mathcal{I}'', \alpha''$ be as in Proposition \ref{fan}.   We allow implied constants to depend on $\theta,\eta,\rho,\eps$.  Let $\mathcal{A}_1, \mathcal{A}_2$ be two subsets of $\mathcal{I}''$.  Then the number of quadruplets $(I''_1, I''_2, p'_1, p'_2)$ with $I''_1 \in \mathcal{A}_1, I''_2 \in \mathcal{A}_2$, $p'_1, p'_2$ primes in $[P', 2P']$, and $I''_1$ lying within $100 \frac{H}{P'P''}$ of  $\frac{p'_2}{p'_1} I''_2$ is
\begin{equation}\label{mixbound}
 \ll (\# \mathcal{A}_1) (\# \mathcal{A}_2) \frac{H}{X} \left (\frac{P'}{\log P'} \right )^2 +(\# \mathcal{A}_1)^{1/2} (\# \mathcal{A}_2)^{1/2} \left (\frac{P'}{\log P'} \right )^2 \log^{-100} P'.
\end{equation}
\end{lemma}

\begin{proof}  Let $\psi: \R \to \R$ be a non-negative Schwartz function whose Fourier transform $\hat \psi(\xi) \coloneqq \int_\R \psi(x) e(-x\xi)\ dx$ is supported on $[-1,1]$.  Observe that if $(I''_1,I''_2,p'_1,p'_2)$ is a quadruplet of the required form, then
$$ \psi \left ( \frac{X}{H} \left ( \log x_{I''_2} - \log x_{I''_1} + \log p'_2 - \log p'_1 \right ) \right ) \gg 1.$$
Thus it will suffice to bound the expression
$$ \sum_{I''_1 \in \mathcal{A}_1, I''_2 \in \mathcal{A}_2} \sum_{p'_1, p'_2 \in [P', 2P']} \psi \left ( \frac{X}{H} \left ( \log x_{I''_2} - \log x_{I''_1} + \log p'_2 - \log p'_1 \right ) \right )$$
by \eqref{mixbound}.  Using the Fourier inversion formula $\psi(x) = \int_\R \hat \psi(\xi) e(x \xi)\ d\xi$, we can write this expression as
$$ \int_\R \hat \psi( \xi ) \left (\sum_{I'' \in \mathcal{A}_1} e \left ( \frac{X}{H} \cdot \xi \log x_{I''} \right ) \right ) \cdot \overline{
\sum_{I'' \in \mathcal{A}_2} e \left ( \frac{X}{H}\cdot  \xi \log x_{I''} \right )}\cdot \left |\sum_{p' \in [P', 2P']} e \left ( \frac{X}{H} \cdot \xi \log p' \right ) \right |^2\ d\xi,$$
which after a change of variable can be bounded by
$$ \ll \frac{H}{X} \int_{|\xi| \leq \frac{X}{H}} |S_1(\xi)| |S_2(\xi)| |T(\xi)|^2\ d\xi $$
where
$$ S_i(\xi) \coloneqq \sum_{I'' \in \mathcal{A}_i} e( \xi \log x_{I''} ) $$
for $i=1,2$ and
$$ T(\xi) \coloneqq \sum_{p \in [P', 2P']} p^{2\pi i \xi}.$$
From the triangle inequality we have
$$ \sup_{\xi \in \R} |S_i(\xi)| \ll \# {\mathcal A}_i$$
while from the large sieve inequality (Lemma \ref{le:largesievish}) we have
$$ \int_{|\xi| \leq \frac{X}{H}} |S_i(\xi)|^2 \ll \# {\mathcal A}_i \frac{X}{H}.$$
Furthermore from \cite[Lemma 2]{mr-p} we have
$$ T(\xi) \ll \frac{P'}{\log P'} \left ( \frac{1}{1+|\xi|} + \log^{-100} P' \right ).$$
for $|\xi| \leq \frac{X}{H}$.  The claim now follows from the triangle inequality and the Cauchy-Schwarz inequality.
\end{proof}

Using this lemma, we have the following tool for converting local approximate constancy to global approximate constancy. The corollary will allow us to show that many of the intervals $I''$ in Proposition~\ref{local} share essentially same values of $T$ and $q$.

\begin{corollary}[Approximate ergodicity]\label{approx} Let $\theta,\eta,X, H, f, \rho,\eps, P', P'', \mathcal{I}'', \alpha''$ be as in Proposition \ref{fan}.   We allow implied constants to depend on $\theta,\eta,\rho,\eps$.  Let $M, K, \delta > 0$.  Let $(Z,d)$ be a metric space, and let $r > 0$ be a radius with the property that every ball of radius $5r/2$ can contain at most $M$ disjoint balls of radius $r/2$.  For each $I'' \in \mathcal{I}''$, let $\mathcal{F}(I'')$ be a finite subset of $Z$ with cardinality at most $K$.  Let $\mathcal{S}$ be a collection of sextuples $(I''_1, I''_2, z_1, z_2, p'_1, p'_2)$ with $I''_1, I''_2 \in \mathcal{I}''$ with $z_1 \in \mathcal{F}(I''_1), z_2 \in \mathcal{F}(I''_2), d(z_1,z_2) \leq r$, and $p'_1, p'_2$ distinct primes in $[P', 2P']$ with $I''_1$ lying within $100 \frac{H}{P'P''}$ of  $\frac{p'_2}{p'_1} I''_2$.  Suppose that
\begin{equation}\label{scard}
 \# {\mathcal S} \geq \delta (X/H) (P'/\log P')^2.
\end{equation}
Then either
\begin{equation}\label{mkd}
 \frac{MK^3}{\delta} \gg \log^{100} P'
\end{equation}
or else there exists $z_0 \in Z$ and a collection ${\mathcal T}$ of pairs $(I'',z)$ with $I'' \in \mathcal{I}$, $z \in \mathcal{F}(I'')$, and $d(z,z_0) \leq 2r$ such that
$$ \# {\mathcal T} \gg \frac{\delta}{MK^3} \frac{X}{H},$$
and such that there are $\gg \frac{\delta^2}{M K^4} \frac{X}{H} (P'/\log P')^2$ sextuples $(I''_1, I''_2, z_1, z_2, p'_1, p'_2) \in \mathcal{S}$ such that $(I''_1,z_1)$, $(I''_2,z_2)$ both lie in ${\mathcal T}$.
\end{corollary}

\begin{proof}  For technical reasons we first need to refine the set ${\mathcal S}$.  Let ${\mathcal T}_0$ be the set of all pairs
$(I''_1,z_1)$ with $I''_1 \in \mathcal{I}''$ and $z_1 \in \mathcal{F}(I'')$. From \eqref{scard} we have
$$ \sum_{(I''_1,z_1) \in \mathcal{T}_0} N(I''_1,z_1) \geq \delta (X/H) (P'/\log P')^2$$
where
$$ N(I''_1,z_1) \coloneqq \# \{ (I''_2,z_2,p'_1,p'_2): (I''_1, I''_2, z_1, z_2, p'_1, p'_2) \in \mathcal{S} \}.$$
We have $\# \mathcal{T}_0 \leq 10 K X/H$.  We conclude that there is a subset $\mathcal{T}_1$ of $\mathcal{T}_0$ with
\begin{equation}\label{nape}
 N(I''_1,z_1) \gg \frac{\delta}{K} (P'/\log P')^2
\end{equation}
for all $(I''_1,z_1) \in {\mathcal T}_1$, such that 
\begin{equation}\label{scard-2}
 \sum_{(I''_1,z_1) \in {\mathcal T}_1} N(I''_1,z_1) \gg \delta (X/H) (P'/\log P')^2.
\end{equation}

Let $\Omega$ be a maximal $r$-separated net in $Z$, thus every point in $Z$ lies within distance $r$ of at least one point in $\Omega$.  From \eqref{scard-2} and the triangle inequality we conclude that
\begin{equation}\label{xa}
 \sum_{z_0 \in \Omega}  
\sum_{\substack{(I''_2,z_2) \in \mathcal{T}_0: z_2 \in B(z_0,2r) \\ (I''_1,z_1) \in \mathcal{T}_1: z_1 \in B(z_0,r)}} \sum_{\substack{p'_1,p'_2 \in [P', 2P'] \\ \mathrm{dist}(I''_1, \frac{p'_2}{p'_1} I''_2) \leq 100 \frac{H}{P' P''}}} 1 \gg \delta (X/H) (P'/\log P')^2.
\end{equation}
If we define 
$$\mathcal{A}_1(z_0) \coloneqq \{ I''_1 \in \mathcal{I}'': \exists z_1 \in B(z_0,r) \text{ such that } (I''_1,z_1) \in \mathcal{T}_1 \}$$
and
$$\mathcal{A}_2(z_0) \coloneqq \{ I''_2 \in \mathcal{I}'': \exists z_2 \in B(z_0,2r) \text{ such that } (I''_2,z_2) \in \mathcal{T}_0 \}$$
then the left-hand side of \eqref{xa} is bounded by
$$ K^2 
\sum_{z_0 \in \Omega} \sum_{\substack{I''_1 \in \mathcal{A}_1(z_0) \\ I''_2 \in \mathcal{A}_2(z_0)}} \sum_{\substack{p'_1,p'_2 \in [P', 2P'] \\ \mathrm{dist}(I''_1, \frac{p'_2}{p'_1} I''_2) \leq 100 \frac{H}{P' P''}}} 1 $$
which by Lemma \ref{mixlem} is bounded by
\begin{align*}\ll K^2 \left (\frac{P'}{\log P'} \right )^2 \Bigl( \sum_{z_0 \in \Omega} & (\# \mathcal{A}_1(z_0)) (\# \mathcal{A}_2(z_0)) \frac{H}{X}  \\ & + \sum_{z_0 \in \Omega} (\# \mathcal{A}_1(z_0))^{1/2} (\# \mathcal{A}_2(z_0))^{1/2} \log^{-100} P' \Bigr).
  \end{align*}
Any pair $(I''_2, z_2) \in \mathcal{T}_0$ can contribute to $\mathcal{A}_2(z_0)$ only if $B(z_0,r/2)$ is contained in $B(z_2,5r/2)$.  As the balls $B(z_0,r/2)$ with $z_0 \in \Omega$ are disjoint, we conclude that each such pair contributes to at most $M$ sets $\mathcal{A}_2(z_0)$, and hence
$$ \sum_{z_0 \in \Omega} \# \mathcal{A}_2(z_0) \ll M K \frac{X}{H}$$
and similarly
$$ \sum_{z_0 \in \Omega} \# \mathcal{A}_1(z_0) \ll M K \frac{X}{H}.$$
By Cauchy-Schwarz, we may thus bound the left-hand side of \eqref{xa} by
$$\ll M K^3 \frac{X}{H} \left (\frac{P'}{\log P'} \right )^2 \left (\sup_{z_0 \in \Omega} \# \mathcal{A}_1(z_0) \frac{H}{X} + \log^{-100} P' \right )$$
and hence
$$ \sup_{z_0 \in \Omega} \# \mathcal{A}_1(z_0) \frac{H}{X} + \log^{-100} P' \gg \frac{\delta}{M K^3}.$$
Thus, either \eqref{mkd} holds, or there exists $z_0 \in \Omega$ with
$$\# \mathcal{A}_1(z_0) \gg \frac{\delta}{M K^3} \frac{X}{H}.$$
Suppose the latter claim is true.  If we now let ${\mathcal T}_2$ denote the collection of those $(I''_1,z_1) \in \mathcal{T}_1$ with $I''_1 \in \mathcal{A}_1(z_0)$ and $z_1 \in B(z_0,r)$, then we have
$$ \# {\mathcal T}_2 \gg \frac{\delta}{M K^3} \frac{X}{H}.$$
From \eqref{nape} there exist $\gg \frac{\delta}{M K^3} \frac{X}{H} \frac{\delta}{K} (P'/\log P')^2$
sextuples $(I''_1, I''_2, z_1, z_2, p'_1, p'_2) \in \mathcal{S}$ such that $(I''_1,z_1) \in {\mathcal T}_2$.  Since $z_1 \in B(z_0,r)$ and $d(z_1,z_2) \leq r$, we have $z_2 \in B(z_0,2r)$.  Thus, if we take ${\mathcal T}$ to be the collection of those $(I''_1,z_1) \in \mathcal{T}_0$ with $I''_1 \in \mathcal{A}_2(z_0)$ and $z_1 \in B(z_0,2r)$, we obtain the claim.
\end{proof}

Let $\theta,\eta,X, H, f, \eps, \rho, P', P'', \mathcal{I}'', \alpha''$ be as in Proposition \ref{fan}. Let $\delta>0$ be a small quantity (depending on $\theta,\eta,\eps$) which we will specify in a moment.  Inspired by Proposition \ref{local}, define a \emph{good quadruple} to be a quadruple $(I'', T, q, a)$, where $I''$ is an interval in $I'' \in \mathcal{I}''$, $T$ is a real number with
\begin{equation}\label{to}
 |T| \leq \frac{1}{\delta} \frac{X^2}{H^{2 - \rho}},
\end{equation}
$q$ is a natural number with
$ 1 \leq q \leq H^{\rho}/\delta$,
$a \in \{0, \dotsc, q-1\}$ is coprime to $q$, and there exists a real number $\theta$ with 
$|\theta| \leq \frac{1}{\delta} \frac{1}{H^{1 - \rho}}$
such that
\begin{equation}\label{dub}
 \alpha''_{I''} = \frac{T}{x_{I''}} + \frac{a}{q} \prod_{p'' \in {\mathcal P}} p'' + \theta \quad \hbox{ mod } \prod_{p'' \in \mathcal{P}} p''
\end{equation}
for a set $\mathcal{P}$ of primes in $[P''/2, P'']$ of cardinality at least $\delta \frac{P''}{\log P''}$.  Proposition \ref{local} guarantees that once $\delta$ is chosen sufficiently small in terms of $\theta, \varepsilon, \eta, \rho$ there exist $\gg X / H $ good quadruples. Throughout we fix $\delta$ sufficiently small so that this holds; in particular, implied constants may now depend on $\delta$ in addition to $\theta,\varepsilon,\eta,\rho$. 

We have some limitations on how many good quadruples can be associated to a single interval $I''$:

\begin{proposition}\label{quip}  Let $\delta, \rho$ be as above, and let $I''$ be an interval in $\mathcal{I}''$.  Let $K \geq \frac{2}{\delta}$, and let
$(I'', T_j, q_j, a_j)$ for $j=1,\dots,K$ be a collection of good quadruples.  Then there exist $1 \leq j < j' \leq K$ with the following properties:
\begin{itemize}
\item[(i)] $q_j = q_{j'}$.
\item[(ii)]  $a_j = a_{j'}$.
\item[(iii)]  $T_j = T_{j'} + O  \left (  \frac{X}{H^{1 - \rho}}  \right )$.
\end{itemize}
\end{proposition}

\begin{proof}  Without loss of generality we may take $K = \lceil \frac{2}{\delta} \rceil$.  For $j = 1, \dotsc, K$, let $\mathcal{P}_j$ be the set of primes in $[P''/2, P'']$ associated to the good quadruple $(I'',T_j,q_j,a_j)$. Then
$$ \sum_{p'' \in [P''/2, P'']} \left ( \sum_{j=1}^K 1_{p'' \in \mathcal{P}_j}  \right ) \geq K \delta \frac{P''}{\log P''} \geq 2 \frac{P''}{\log P''}$$
and $\sum_{j=1}^K 1_{\mathcal{P}_j}  \leq K \ll 1/\delta$.  From this and the prime number theorem we conclude that $\sum_{j=1}^K 1_{\mathcal{P}_j} \geq 2$ for at least $\gg \frac{P''}{\log P''}$ primes in $[P', 2P']$; this implies that there exist distinct indices $j, j' \in \{1, \dotsc, K\}$ such that
$$ \# (\mathcal{P}_j \cap \mathcal{P}_{j'}) \gg \frac{P''}{\log P''}.$$
If one writes $Q \coloneqq \prod_{p'' \in \mathcal{P}_j \cap \mathcal{P}_{j'}} p''$, we then have
\begin{equation}\label{qq}
Q \gg \exp( c_\delta P'' ) \geq \exp( c_\delta X^{\eps^2} )
\end{equation}
for some $c_\delta > 0$.  On the other hand, from \eqref{dub} one has
\begin{equation}\label{a1}
 \alpha''_{I''} = \frac{T_j}{x_{I''}} + \frac{a_j}{q_j} Q + O \left ( \frac{1}{H^{1 - \rho}} \right ) \hbox{ mod } Q
\end{equation}
and
\begin{equation}\label{a2}
 \alpha''_{I''} = \frac{T_{j'}}{x_{I''}} + \frac{a_{j'}}{q_{j'}} Q + O \left ( \frac{1}{H^{1 - \rho}} \right ) \hbox{ mod } Q.
\end{equation}
In particular,
$$ (\frac{a_j}{q_j} - \frac{a_{j'}}{q_{j'}}) Q = O \left ( \frac{X^2}{H^{2 - \rho}} \right ) \hbox{ mod } Q $$
which when combined with \eqref{qq} (and noting that the denominator on the left-hand side is at most $O_\delta(H^{2 \rho})$) forces
$$ \frac{a_j}{q_j} - \frac{a_{j'}}{q_{j'}} = 0 \hbox{ mod } 1.$$
Since $a_j/q_j$ and $a_{j'}/q_{j'}$ are in lowest terms and in $[0, 1)$, this implies that $a_j = a_{j'}$ and $q_j = q_{j'}$.  Subtracting \eqref{a1} from \eqref{a2}, we conclude that
$$ \frac{T_j - T_{j'}}{x_{I''}} = O \left ( \frac{1}{H^{1 - \rho}} \right ) \hbox{ mod } Q;$$
since $|T_j - T_{j'}| \leq \frac{2}{\delta} \frac{X^2}{H^{2 - \rho}}$, we conclude from \eqref{qq} that
$$ \frac{T_j - T_{j'}}{x_{I''}} = O \left ( \frac{1}{H^{1 - \rho}} \right ),$$
and hence $T_j - T_{j'} \ll_\delta \frac{X}{H^{1 - \rho}}$.  The claim follows.
\end{proof}

From the above proposition and the greedy algorithm, we conclude

\begin{corollary}\label{greedy}  For each $I'' \in \mathcal{I}''$, there exists a set $\mathcal{F}(I'')$ of triples $(T,q,a)$ of cardinality
$$ \# \mathcal{F}(I'') \leq \frac{2}{\delta},$$
such that, for any good quadruple $(I'',T,q,a)$, there exists $T' \in \mathbb{R}$ such that $(T',q,a) \in \mathcal{F}(I'')$ and
$$ T = T' + O \left ( \frac{X}{H^{1 - \rho}} \right ).$$
\end{corollary}

On the other hand, Proposition \ref{local} provides us with a large number of quadruples:

\begin{proposition}\label{clack} Let $\delta$ be as above and $X$ sufficiently large depending on $\delta$ and $\eps$.  All implied constants may depend on $\eps, \eta, \theta, \rho$.  Then, for $\gg(X/H) \cdot (P' / \log P')^2$ of the pairs $(I''_1, I''_2)$ of intervals $(\mathcal{I}'')^2$, there exist $T_1,T_2,q',a'_1,a'_2$ such that $(T_1,q', a'_1) \in \mathcal{F}(I''_1)$ and $(T_2, q', a'_2) \in \mathcal{F}(I''_2)$, and 
\begin{equation}\label{t10^4}
 T_2 = T_1 + O \left ( \frac{X}{H^{1 - \rho}} \right ).
\end{equation}
Furthermore, for each such pair, there exist primes $p'_1, p'_2 \in [P', 2P']$ coprime to $q'$ such that $I''_1$ lies within $100 \frac{H}{P'P''}$ of  $\frac{p'_2}{p'_1} I''_2$, and such that 
\begin{equation}\label{pio-qp}
 p'_2 a_1' = p'_1 a_2' \hbox{ mod } q'.
\end{equation}
\end{proposition}

\begin{proof}  This is almost immediate from Proposition \ref{local}; the main difficulty is that the integers $a,q$ provided by that proposition need not be coprime.

We resolve this as follows.  If $I'' \in \mathcal{I}''$ and $(T,q,a) \in \mathcal{F}(I'')$, then $q$ has at most $O( \frac{ \log X }{\log P''} ) = O_\eps(1)$ prime factors in $[P''/2,P'']$.  Thus, for each $I'' \in \mathcal{I}''$, there are at most $O(1)$ primes that divide $q$ for some $(T,q,a) \in \mathcal{F}(I'')$.

Proposition \ref{local} provides us with $\gg \frac{X}{H} \left(\frac{P'}{\log P'}\right)^2$ pairs $(I''_1, I''_2)$ of intervals $(\mathcal{I}'')^2$, together with associated primes $p'_1, p'_2$, obeying the properties of that proposition.  It could happen that $p'_1$ or $p'_2$ divides $q$ for some $(T,q,a)$ in $\mathcal{F}(I''_1)$ or $\mathcal{F}(I''_2)$, but by the preceding paragraph, the number of times this can happen is at most $O( \frac{X}{H} \frac{P'}{\log P'} )$, which is a negligible portion when $X$ is large enough.  Thus for $\gg \frac{X}{H} \left(\frac{P'}{\log P'}\right)^2$ of the above pairs, $p'_1$ or $p'_2$ do not divide any such $q$.

From Proposition \ref{local}, we have
$$ \alpha''_{I''_j} = \frac{T}{x_{I''_j}} + \frac{a_j}{q} Q + O \left ( \frac{1}{H^{1 - \rho}} \right ) \hbox{ mod } Q $$
for $j=1,2$, where $Q \coloneqq \prod_{p'' \in {\mathcal P}(I''_1, I''_2)}  p''$.  We write $a_1/q$ in lowest terms as $a'_1/q'$.  Then $(I''_1,T,q',a'_1)$ is a good quadruple and $p'_1,p'_2$ do not divide $q'$.  From \eqref{pio-q} we may thus also write $a_2/q$ in lowest terms as $a'_2/q'$ and still have that \eqref{pio-qp} holds.  Then $(I''_2,T,q',a'_2)$ is a good quadruple, and the claim follows from Corollary \ref{greedy}.
\end{proof}
 
Let $Z$ be the collection of triples $(T,q,a)$ with $T \in \R$, $q \geq 1$, and $a$ coprime to $q$, endowed with the metric\footnote{The $\frac{1}{100} 1_{a_1 \neq a_2}$ term is present only to keep the metric $Z$ from degenerating, but otherwise plays no role in the argument; if one prefers, one could drop this term and observe that Corollary \ref{approx} also applies to degenerate metric spaces.}
\begin{equation}\label{metric-def}
 d((T_1,q_1,a_1), (T_2,q_2,a_2)) \coloneqq  c(\delta) \frac{H^{1 - \rho}}{X} |T_1-T_2| + 1_{q_1 \neq q_2} + \frac{1}{100} 1_{a_1 \neq a_2}.
\end{equation}
and some sufficiently small constant $c(\delta) > 0$ depending on $\delta$ (and thus ultimately on $\theta, \eta, \rho, \eps$). Let $\mathcal{S}$ be the collection of sextuples
$$ (I''_1, I''_2, (T_1, q', a_1), (T_2,q',a_2), p'_1, p'_2)$$
with $I''_1, I''_2 \in \mathcal{I}''$, $(T_1,q',a_1) \in \mathcal{F}(I''_1)$, $(T_2,q',a_2) \in \mathcal{F}(I''_2)$, and $p'_1,p'_2$ distinct primes in $[P', 2P']$ with $I''_1$ lying within $100 \frac{H}{P'P''}$ of  $\frac{p'_2}{p'_1} I''_2$, with $p_1,p_2$ coprime to $q'$ and obeying \eqref{pio-qp} and \eqref{t10^4}.  In particular (for $c(\delta)$ sufficiently small) we have
$$d((T_1,q',a_1),(T_2,q',a_2)) \leq \frac{1}{10}.$$
From Proposition \ref{clack} we have $\# \mathcal{S} \gg (X/H) \cdot (P' / \log P')^2$.  Applying Corollary \ref{approx} with $r = \frac{1}{10}$, $M = 100$, $K = \frac{2}{\delta}$, we conclude that there exists $(T_0,q_0,a_0) \in Z$ and a collection ${\mathcal T}$ of quadruples $(I'',T,q,a)$ with $I'' \in \mathcal{I}$, $(T,q,a) \in \mathcal{F}(I'')$, and $d((T,q,a),(T_0,q_0,a_0)) \leq \frac{1}{5}$ such that
\begin{equation}\label{calt}
 \# {\mathcal T} \gg \frac{X}{H},
\end{equation}
and there are $\gg \frac{X}{H} (P'/\log P')^2$ sextuples $(I''_1, I''_2, (T_1,q_1,a_1), (T_2,q_2,a_2), p'_1, p'_2) \in \mathcal{S}$ such that $(I''_1,T_1,q_1,a_1)$, $(I''_2,T_2,q_2,a_2)$ both lie in ${\mathcal T}$.

If $(I'',T,q,a) \in {\mathcal T}$, then $d((T,q,a),(T_0,q_0,a_0)) \leq \frac{1}{5}$, and hence by \eqref{metric-def} we have $q=q_0$ and
\begin{equation}\label{t0}
 T = T_0 + O \left ( \frac{X}{H^{1 - \rho}} \right ).
\end{equation}  
From \eqref{to} we thus have
\begin{equation}\label{t00}
 T_0 \ll \frac{X^2}{H^{2 - \rho}}.
\end{equation}

At present $q_0$ obeys the bounds $1 \leq q_0 \ll H^{\rho}$.  We can improve the control on $q_0$ significantly. 

\begin{proposition}\label{qb-improv}  $q_0 \ll 1$.
\end{proposition}

\begin{proof}   Consider the graph ${\mathcal G}$ whose vertex set $V$ is the set ${\mathcal T}$ as above, and whose edge set $E$ consists of pairs
$(I''_1,T_1,q_0,a_1)$, $(I''_2,T_2,q_0,a_2)$ in ${\mathcal T}$ with  $$(I''_1, I''_2, (T_1,q_0,a_1), (T_2,q_0,a_2), p'_1, p'_2) \in \mathcal{S}$$ 
for some $p_1'$ and $p_2'$. Then by the preceding dicussion $G$ has $\gg N$ vertices and $\gg dN$ edges, where $N \coloneqq X/H$ and $d \coloneqq (P'/\log P')^2$.

Now let $k$ be the first even integer for which
$$ d^k \geq N^{2 + \varepsilon}.$$
Using the Blakley-Roy inequality as in Section \ref{local-sec}, the number of $(\frac{k}{2} + 1)$-tuples
$$ (Q_0, \dots,Q_{k/2}) \in V^{k/2+1}$$
such that $\{ Q_j, Q_{j+1} \} \in E$ for $0 \leq j < k/2$ is $\gg d^{k/2} N$.  The number of possible values for the pair $(Q_0, Q_{k/2})$ is $O(N^2)$.  Thus by the Cauchy-Schwarz inequality, there are $\gg d^k$ pairs of $\frac{k+2}{2}$-tuples of the above form with matching pairs
$(Q_0, Q_{k/2})$.  Relabeling, we conclude that there the number of $k$-tuples 
$$ (Q_j)_{j=0,1,\ldots, k-1} \in V^k$$
such that $\{ Q_{j}, Q_{j+1} \} \in E$ for $j=0,1,\ldots, k - 1$ is $\gg d^k$.  On the other hand, we may upper bound the number of such tuples in a different way, as we will now do.  Writing $Q_j = (I''_j, T_j, q_0, a_j)$, we see from \eqref{pio-qp} that there are primes $p'_{j,1}, p'_{j,2} \in [P', 2P']$ such that
$$
 p'_{j,2} a_{j} = p'_{j,1} a_{j+1} \hbox{ mod } q_0
 $$
(with the periodic convention $a_k = a_0$)
and such that $I''_j$ lies within $100 \frac{H}{P'P''}$ of  $\frac{p'_{j,2}}{p'_{j,1}} I''_{j+1}$ for all $j = 0,1,\ldots, k - 1$.  From the first claim we have
$$ \prod_{j=1}^k p'_{j,2} = \prod_{j=1}^k p'_{j,1} \hbox{ mod } q_0,$$
while from the second claim we have
$$ \prod_{j=1}^k p'_{2,j} - \prod_{j=1}^k p'_{1,j} \ll \frac{(P')^k}{N}.$$
by repeating the derivation of \eqref{qb}.  By Lemma \ref{nte}, the number of tuples of primes $(p'_{1,1},\dots,p'_{k,1},p'_{1,2},\dots,p'_{k,2})$ obeying these constraints is $\ll \frac{d^k}{N} ( \frac{1}{q_0^{1/2}} + \frac{1}{\log X} ) )$.  There are $\ll N$ choices for $I''_1$, and this interval and the tuple of primes determine all the other $I''_k$.  Since all the sets ${\mathcal F}(I''_j)$ have cardinality $O_\delta(1)$, we conclude that the number of $k$-tuples $(Q_j)_{j=0,1,\ldots, k- 1}$ under consideration is
$$ \ll N \frac{d^k}{N} \left ( \frac{1}{q_0^{1/2}} + \frac{1}{\log X} \right ).$$
Comparing the upper and lower bounds yields
$$ \frac{1}{q_0^{1/2}} + \frac{1}{\log X}  \gg 1$$
and the claim follows.
\end{proof}

From \eqref{dub}, \eqref{t0} we see that whenever $(I'',T,q_0,a) \in {\mathcal T}$, one has
\begin{equation}\label{alph}
 \alpha''_{I''} = \frac{T_0}{x_{I''}} + \frac{b}{q_0} + O\left ( \frac{1}{H^{1 - \rho}} \right ) \hbox{ mod } 1
\end{equation}
for some $b \in \Z/q_0\Z$.  Since each $I''$ is associated to $O(1)$ quadruples in ${\mathcal T}$, we conclude from \eqref{calt} that for $\gg_{\eps,\delta} X/H$ intervals $I'' \in \mathcal{I}''$, one has \eqref{alph} for some $b \in \Z/q_0\Z$.  

Let $I''$ be one of these intervals, so that (see \eqref{hpp})
$$ \left |\sum_{n \in I''} f(n) e(-\alpha''_{I''} n) \right | \gg \frac{H}{P' P''}.$$
Let $H^* \coloneqq \frac{H^{1 - 2 \rho}}{P'P''}$.  We may translate $I''$ by any shift of size at most $H^*$ without affecting this estimate.  Averaging over such shifts, we conclude that
$$ \left |\int_{I''} \sum_{x < n \leq x+H^*} f(n) e(-\alpha''_{I''} n)\ dx \right | \gg \frac{H}{P'P''} H^*$$
and thus by the triangle inequality
$$ \int_{I''} \left |\sum_{x < n \leq x+H^*} f(n) e(-\alpha''_{I''} (n-x) - bx/q_0) \right |\ dx \gg \frac{H}{P'P''} H^*$$
From \eqref{alph}, \eqref{t00} and Taylor expansion, we have
\begin{align*}
 e(-\alpha''_{I''} (n-x) - bx/q_0) &= e \left ( - \frac{T_0}{x} (n-x) \right ) e(b n/q_0) + O( H^{-\rho} )  \\
&= n^{-2\pi i T_0} x^{2\pi i T_0} e(bn/q_0) + O( H^{-\rho} ) .
\end{align*}
The contribution of the $O(H^{-\rho})$ is negligible, thus 
$$ \int_{I''} \left |\sum_{x < n \leq x+H^*} f(n) n^{-2\pi i T_0} e(bn/q_0) \right |\ dx \gg \frac{H}{P'P''} H^*.$$
Recalling $(b, q_0) = 1$ and Proposition \ref{qb-improv}, we can apply a Fourier decomposition 
$$ e(bn/q_0) = \sum_{q_0 = q_1 q_2} \sum_{\chi\ (q_1)} c_{b,\chi} 1_{q_2|n} \chi(n/q_2) \ , \ c_{b, \chi} := \frac{1}{\varphi(q_1)} \sum_{x \text{ mod } q_1} \overline{\chi}(x) e \Big ( \frac{b x}{q_1} \Big )$$
where $c_{b,\chi} \ll 1$ and $\chi$ ranges over Dirichlet characters of modulus $q_1$. From the triangle inequality, we thus have
$$ \sum_{q_0=q_1 q_2} \sum_{\chi\ (q_1)}
\int_{I''} \left |\sum_{x < n \leq x+H^*} f(n) n^{-2\pi i T_0} 1_{q_2|n} \chi(n/q_2) \right |\ dx \gg \frac{H}{P'P''} H^*.$$
Summing over the $\gg_\eps X/H$ intervals $I''$, we conclude that
$$ \sum_{q_0=q_1 q_2} \sum_{\chi\ (q_1)}
\int_{X/10P'P''}^{10X/P'P''} \left |\sum_{x < n \leq x+H^*} f(n) n^{-2\pi i T_0} 1_{q_2|n} \chi(n/q_2) \right |\ dx \gg \frac{X}{P'P''} H^*.$$
By the triangle inequality, there thus exist $q_0=q_1q_2$ and $\chi\ (q_1)$ such that
$$
\int_{X/10P'P''}^{10X/P'P''} \left | \sum_{x < n \leq x+H^*} f(n) n^{-2\pi i T_0} 1_{q_2|n} \chi(n/q_2) \right |\ dx \gg \frac{X}{P'P''} H^*.$$
Writing $n = d m$ with $d | q_2^{\infty}$ and $(m,q_2) = 1$ we obtain by the triangle inequality
$$
\sum_{\substack{d | q_2^{\infty} \\ q_2 \mid d}} \int_{X/10 P' P''}^{10 X / P' P''} \left | \sum_{\substack{x < d n \leq x + H^{*} \\ (n,q_2) = 1}} f(n) n^{-2 \pi i T_0} \chi(n) \right | dx \gg \frac{X}{P' P''} H^*. 
$$
where $d | q_2^{\infty}$ means that all the prime factors of $d$ are also prime factors of $q_2$. Since $\sum_{d | q_2^{\infty}} d^{-1} \ll 1$ there exists an natural number $d = O(1)$ such that,
$$
\int_{X / 10 d P' P''}^{10 X / d P' P''} \left | \sum_{\substack{x < n \leq x + H^{*} / d \\ (n,q_2) = 1}} f(n) n^{-2 \pi i T_{0}} \chi(n) \right | dx \gg_{\eta, \eps, \delta} \frac{X}{P' P''} H^*. 
$$
Therefore by \cite[Proposition A.3]{MRT} we have $\mathbb{D}(f 1_{(n,q_2) = 1} n^{-2\pi i T_0} \chi; T'; Q) \ll 1$ for some $Q \ll 1$ and $|T'|\ll X$. Therefore $\mathbb{D}(f; T; Q) \ll 1$ for some $|T| \ll X^2 / H^{2 - \rho}$ and $Q \ll 1$ as claimed.


\section{Proof of Corollary \ref{maincor} and Corollary \ref{cor:correlations}}
\label{se:triplecorr}

Now we prove Corollary \ref{maincor} and Corollary \ref{cor:correlations}.
It is enough to prove the former corollary since, for any fixed $Q > 0$ and $A > 0$, we have $\mathbb{D}(\lambda; X^{A}; Q) \rightarrow \infty$ as $X \rightarrow \infty$ by the Vinogradov-Korobov zero-free region \cite[\S 9.5]{Montgomery}.  

We restrict attention to the correlation for $f(n) a(n+h) b(n+2h)$, as the other two correlations are handled similarly.  The proof proceeds along classical lines by noticing that
\begin{align} \nonumber
\sum_{|h| \leq H} & \left (1 - \frac{|h|}{H} \right ) f(n) a(n + h) b(n + 2 h) \\ &  = \frac{1}{H} \int_{1}^{X}
\int_{0}^{1} S_{x, f}(\alpha) S_{x, b}(\alpha) S_{x, a}(- 2\alpha) d \alpha dx + O(H)
\end{align}
where
$$
S_{x, g}(\alpha) \coloneqq \sum_{x < n \leq x + 2 H} g(n) e(\alpha n).  
$$
Notice that
\begin{align*} 
&  \left | \int_{0}^{1}  S_{x, f}(\alpha) S_{x, b}(\alpha) S_{x, a}(- 2 \alpha) d \alpha \right | \leq \sup_{\alpha} |S_{x, f}(\alpha)|^{1/3} \int_{0}^{1} |S_{x,b}(\alpha)| \cdot |S_{x, a}(-2 \alpha)| \cdot |S_{x, f}(\alpha)|^{2/3} \ d \alpha \\ \nonumber 
	& \quad\quad \leq \sup_{\alpha} |S_{x, f}(\alpha)|^{1/3} \cdot \left ( \int_{0}^{1} |S_{x, b}(\alpha)|^3 d \alpha \right )^{1/3} \cdot \left ( \int_{0}^{1} |S_{x, a}(\alpha)|^3 d \alpha \right )^{1/3} \cdot \left ( \int_{0}^{1} |S_{x, f}(\alpha)|^2 d \alpha \right )^{1/3}.
\end{align*}
We now claim the bound
$$
\int_{0}^{1} |S_{x, a}(\alpha)|^3 d \alpha \ll H^2
$$
If $|a(n)| \ll \Lambda(n)$ then this bound follows from \cite[Proposition 4.2]{GreenTaoRes}. On the other hand if $|a(n)| \ll 1$ for all integers $n \geq 1$, then, by H\"older's inequality,
$$
\int_{0}^{1} |S_{x, a}(\alpha)|^3 d \alpha \leq \left (  \int_{0}^{1} |S_{x, a}(\alpha)|^2 d \alpha \right )^{1/2} \cdot \left ( \int_{0}^{1} |S_{x, a}(\alpha)|^4 d\alpha \right )^{1/2} \ll H^{1/2} \cdot H^{3/2} = H^2.
$$
The general case $a(n) \ll 1 + \Lambda(n)$ now follows from the triangle inequality.  Similarly for $b(n)$.
Therefore,
$$
 \left | \int_{0}^{1} S_{x, f}(\alpha) S_{x, b}(\alpha) S_{x, a}(- 2 \alpha) d \alpha \right |  \ll \sup_{\alpha} |S_{x, f}(\alpha)|^{1/3} \cdot H^{5/3}
 $$
 and finally,
 $$
 \int_{1}^{X} \sup_{\alpha} |S_{x, f}(\alpha)|^{1/3}  d \alpha \leq \left ( \int_{1}^{X} \sup_{\alpha} |S_{x, f}(\alpha)| d\alpha \right )^{1/3} \cdot X^{2/3}.
 $$
 Thus,
 \begin{equation} \label{eq:end}
   \left | \sum_{|h| \leq H} \left (1 - \frac{|h|}{H} \right ) \sum_{n \leq X} f(n) a(n + h) b(n + 2 h) \right | \leq H^{2/3} \cdot X^{2/3} \cdot \left ( \int_{1}^{X} \sup_{\alpha} |S_{x, f}(\alpha)| d \alpha \right )^{1/3}
   \end{equation}
 
   Therefore if the left-hand side of $\eqref{eq:end}$ is $\geq \eta H X$, then,
   $$
   c \eta^{3} H X \leq \int_{1}^{X} \sup_{\alpha} |S_{x,f}(\alpha)| d \alpha
   $$
   for some absolute constant $c > 0$. 
   Hence, for some $Y \in [c \eta^3 X/3, X]$, one has, 
\[
\int_{Y}^{2Y} \sup_{\alpha} |S_{x, f}(\alpha)| d \alpha \geq \frac{c \eta^3}{2} YH. 
\]
Now the claim follows from Theorem \ref{main-thm}.

\bibliography{short-uniformity}
\bibliographystyle{plain}

\end{document}